\documentclass[10pt,reqno]{amsart}
\usepackage{amsmath}
\usepackage{amssymb}
\usepackage{dsfont}
\usepackage[dvips,draft,final]{graphics}
\usepackage{amssymb,amsmath}
\usepackage[T1]{fontenc}
\usepackage{fancyhdr}
\usepackage{url}

\usepackage{color}
\usepackage{graphicx}

\def\squarebox#1{\hbox to #1{\hfill\vbox to #1{\vfill}}}

\newtheorem{thm}{Theorem}[section]
 
\newtheorem{Def}{Definition}[section]
\newtheorem{lem}{Lemma}[section]

\numberwithin{equation}{section}

\newcommand{\bel}{\begin{equation} \label}
\newcommand{\ee}{\end{equation}}

\newcommand{\pd}{\partial}

\newcommand{\R}{\mathbb{R}}

\def\epsilon{\varepsilon}
\def\phi {\varphi}

\newtheorem{prop}{Proposition}[section]

\providecommand{\abs}[1]{\left\lvert#1\right\rvert}
\providecommand{\norm}[1]{\left\lVert#1\right\rVert}
\numberwithin{equation}{section}

\renewcommand{\leq}{\leqslant}
\renewcommand{\geq}{\geqslant}
\providecommand{\abs}[1]{\left\lvert#1\right\rvert}
\providecommand{\norm}[1]{\left\lVert#1\right\rVert}

\def\beq{\begin{equation}}
\def\eeq{\end{equation}}
\newcommand{\bea}{\begin{eqnarray}}
\newcommand{\eea}{\end{eqnarray}}
\newcommand{\beas}{\begin{eqnarray*}}
\newcommand{\eeas}{\end{eqnarray*}}

{

\begin{document}

\title[Determination of  nonlinear term for hyperbolic equations]{On the determination of  nonlinear terms appearing in semilinear hyperbolic equations}

\author{Yavar Kian}
\address{Aix Marseille Univ, Universit\'e de Toulon, CNRS, CPT, Marseille, France.}
\email{yavar.kian@univ-amu.fr}

\begin{abstract}
We consider the inverse problem of determining a general nonlinear term appearing in a semilinear hyperbolic equation on a Riemannian manifold  with boundary $(M,g)$ of dimension $n=2,3$. We prove results of unique recovery of the nonlinear term $F(t,x,u)$, appearing in the equation $\partial_t^2u-\Delta_gu+F(t,x,u)=0$ on $(0,T)\times M$ with $T>0$, from some partial knowledge of the solutions $u$ on the boundary of the time-space cylindrical manifold $(0,T)\times M$ or  on the lateral boundary $(0,T)\times\partial M$. We determine the expression $F(t,x,u)$ both on the boundary $x\in\partial M$ and inside the manifold $x\in M$.

\end{abstract}


\maketitle

\section{Introduction}
\subsection{Statement of the problem}
Let $(M,g)$ be a smooth compact and connected Riemannian manifold with boundary of dimension $n\geq2$ and let $T>0$.  We introduce the Laplace and wave operators
\begin{align}
\label{def_Laplace}
\Delta_{g} u = |g|^{-1/2} \sum_{j,k=1}^n\pd_{x_j}
 \left( g^{jk} |g|^{1/2} \pd_{x_k} u\right),
 \quad 
 \Box_g=\partial_t^2-\Delta_g,
\end{align}
where $|g|$ and $g^{jk}$ denote the absolute of value of the determinant and the inverse of $g$
in local coordinates, and consider, for $T>0$, the semilinear wave equation 
\begin{equation}\label{wave}\Box_g u+F(t,x,u)=0,\quad (t,x)\in (0,T)\times M,\end{equation}
with a nonlinear term $F$ suitably chosen. In this paper, we consider the inverse problem of determining $F$ from observations of solutions of \eqref{wave} on the boundary of the manifold $(0,T)\times M$.
\subsection{Motivations}
Let us first observe that nonlinear wave equations of the form \eqref{wave} can be associated with different models where the transmission of waves  is perturbed by a semilinear expression. Such phenomenon  can occur in many mechanical and electromagnetic models. For instance, we can mention the study of vibrating systems where the expression $F(t,x,u)$ can be seen as a nonlinear  perturbation of the system. The semilinear term $F(t,x,u)$ can also be associated with other perturbations arising
in electronics like  in the telegraph equation or for semi-conductors (see for instance \cite{CH}).
In this context, the goal of our inverse problem is to recover the nonlinear expression $F(t,x,u)$ which describes  the underlying physical law of the perturbed system. 

Beside these physical motivations, we mention that there is a natural mathematical motivation for the study of such inverse problems which are highly nonlinear and ill-posed.

\subsection{Known results}

Let us first mention that, to the best of our knowledge, there is only a small number of papers dealing with inverse problems for nonlinear partial differential equations. Among them we can mention the work  \cite{I3,I4,I5}  of Isakov dedicated to the recovery of nonlinear terms appearing in elliptic or parabolic equations. The method developed by Isakov is based on a linearization of the inverse problem for nonlinear equations and results based on recovery of coefficients for linear equations. This  approach has been applied in different other context. For instance, we can mention the work of \cite{IN,MU,SuU}, dealing with the unique recovery of nonlinear terms appearing in  nonlinear elliptic equations and the work of \cite{CK} dealing with the stable recovery of a semilinear term appearing in a parabolic equation. For more specific nonlinear terms, we can mention the work of \cite{COY,Kl}, who have considered similar problems with single measurements.

For hyperbolic equations we refer to the work of \cite{NM,NV} dealing with the recovery of a conductivity and  quadratic coefficients appearing in a non-linear wave equation of divergence form. We mention also the recent works of \cite{HUW,KLU,KLOU}, who have considered inverse problems for semilinear hyperbolic equations on a general Lorentzian manifold. To the best of our knowledge, beside the present paper, the recovery of a general nonlinear term, appearing in  hyperbolic equations, from boundary measurements has not been addressed so far.
\subsection{Preliminary results}

Before the statement of our main result let us first state some  properties of solutions of \eqref{wave}, that will be required in our analysis. Let us first fix the class of nonlinear terms under consideration. Let $b>0$ be such that, for $n=2$, $b>1$ and, for $n=3$, $b\in\left(1,\frac{13}{3}\right]$. For $c_1>0$ a fixed constant, we consider $\mathcal A$ the set of functions $F\in \mathcal C^3(\R_+\times M\times\R)$ satisfying
\bel{non1}|\partial_t^k\partial_x^\alpha \partial_u^jF(t,x,u)|\leq c_1(1+|u|^{b-j}),\quad (t,x,u)\in\R_+\times M\times\R,\ k+|\alpha|+j\leq 3,\ee
\bel{comcom1} \partial_t^kF(0,x,u)=0,\quad x\in\partial M,\ u\in\R,\ k=0,1.\ee
We fix also the set $\mathcal A_*$ of functions $F\in \mathcal C^3(\R_+\times M\times\R)$ satisfying \eqref{non1} and
\bel{comcom2} \partial_t^kF(0,x,0)=0,\quad x\in\partial M,\ k=0,1.\ee
For any $T>0$, we fix also $\mathcal H(0,T)$ the space of elements $$G=(f,u_0,u_1)\in H^{\frac{11}{2}}((0,T)\times\partial M)\times H^{\frac{11}{2}}( M)\times H^{\frac{9}{2}}( M)$$
satisfying the compatibility conditions
\bel{comp1}f_{|t=0}={u_0}_{|\partial M},\quad\partial_tf_{|t=0}={u_1}_{|\partial M},\quad\partial_t^2f_{|t=0}={\Delta_g u_0}_{|\partial M},\quad\partial_t^3f_{|t=0}={\Delta_g u_1}_{|\partial M},\quad\partial_t^4f_{|t=0}={\Delta_g^2 u_0}_{|\partial M}.\ee
Then, for $n=2,3$, $F\in\mathcal A$,  $T'>0$, $(f,u_0,u_1)\in\mathcal H(0,T')$ and $T\leq T'$ we consider the following problem
\bel{eqn1}
\left\{ \begin{array}{ll}  \pd_t^2 u-\Delta_{g} u + F(t,x,u)  =  0, & \mbox{in}\ (0,T)\times M,\\  
u  =  f, & \mbox{on}\ (0,T)\times \partial M,\\  
 u(0,\cdot)  =u_0,\quad  \pd_tu(0,\cdot)  =u_1 & \mbox{in}\ M.
\end{array} \right.
\ee
We prove in the Appendix (see Lemma \ref{l2}), that for 
$$\norm{f}_{H^{\frac{5}{2}}((0,T')\times\partial M)}+\norm{u_0}_{H^{\frac{5}{2}}( M)}+\norm{u_1}_{H^{\frac{3}{2}}( M)}\leq L$$
and for 
\bel{pp} \left\{\begin{array}{ll} p>\max(b,3(b-1))&\ \textrm{if }n=2\\ p=5&\ \textrm{if }n=3\end{array}\right.\ee
 there exists $T_1(L)\in(0,T']$ such that, for all $0<T< T_1(L)$ and all $F\in\mathcal A$, the problem \eqref{eqn1} admits a unique solution $u\in  W^{1,\frac{p}{b-1}}(0,T;H^2(M))\cap W^{3,\frac{p}{b-1}}(0,T;L^2(M))$.

We consider  also $\mathcal H_*(0,T')$ the space of elements $f\in H^{\frac{11}{2}}((0,T')\times\partial M)$
satisfying the compatibility conditions
\bel{comp2}f_{|t=0}=\partial_tf_{|t=0}=\partial_t^2f_{|t=0}=\partial_t^3f_{|t=0}=\partial_t^4f_{|t=0}=0.\ee
 In the same way, we prove in the Appendix (see Lemma \ref{ll2}) that for
$$\norm{f}_{H^{\frac{5}{2}}((0,T')\times\partial M)}\leq L,$$
$p$ satisfying \eqref{pp}, $T_1(L)\in(0,T']$  and all $F\in\mathcal A_*$, the problem \eqref{eqn1}, with $u_0=u_1=0$, admits a unique solution $u\in  W^{1,\frac{p}{b-1}}(0,T;H^2(M))\cap W^{3,\frac{p}{b-1}}(0,T;L^2(M))$.

Then, for $G\in\mathcal H(0,T')$ we denote by $u_{F,G}\in W^{1,\frac{p}{b-1}}(0,T;H^2(M))\cap W^{3,\frac{p}{b-1}}(0,T;L^2(M))$ the solution of \eqref{eqn1}. In the same way, we denote by $u_{F,f}\in W^{1,\frac{p}{b-1}}(0,T;H^2(M))\cap W^{3,\frac{p}{b-1}}(0,T;L^2(M))$ the solution of \eqref{eqn1} with $u_0=u_1=0$. Then, for some $L>0$, $\epsilon\in(0,1)$, fixing $0<T< T_1(L+3\epsilon)$ and the set
$$\mathcal K_L:=\{G\in\mathcal H(0,T'):\ \norm{G}_{H^{\frac{5}{2}}((0,T')\times\partial M)\times H^{\frac{5}{2}}( M)\times H^{\frac{3}{2}}( M)}\leq L+3\epsilon\},$$
we define the boundary maps
$$\mathcal B_{F,\gamma_1}: \mathcal K_L\ni G\longmapsto ({\partial_\nu u_{F,G}}_{|(0,T)\times\gamma_1},u_{F,G}(T,\cdot)_{|M})\in L^2((0,T)\times\gamma_1)\times H^1(M),$$
$$\mathcal N_{F,\gamma_1}:\{h\in\mathcal H_*(0,T'):\ \norm{h}_{H^{\frac{5}{2}}((0,T')\times\partial M)}\leq L+\epsilon\} \ni f\longmapsto {\partial_\nu u_{F,f}}_{|(0,T)\times\gamma_1}\in L^2((0,T)\times\gamma_1),$$
with $\gamma_1$ an open subset of $\partial M$ and $\nu$ the outward unit normal vector to $\partial M$. We prove in Theorem \ref{t2} that the maps $\mathcal B_{F,\gamma_1}$ and $\mathcal N_{F,\gamma_1}$ admit a continuous Fr\'echet derivative denoted by $\mathcal B_{F,\gamma_1}'$ and $\mathcal N_{F,\gamma_1}'$. The observation  of our inverse problem will be given by some partial knowledge of the Fr\'echet derivative of the map $\mathcal B_{F,\gamma_1}$ and $\mathcal N_{F,\gamma_1}$.
\subsection{Main results}
In our first result we consider the recovery of the nonlinear term $F(t,x,u)$ restricted to a portion of the lateral   boundary $(0,T)\times \partial M$. More precisely, we fix $\gamma$ an arbitrary open subset of $\partial M$, $\delta>0$, $\chi\in\mathcal C^\infty_0((0,T'+1)\times \partial M)$ satisfying $\chi=1$ on $[\delta,T']\times\gamma$ and $$\mathcal H_{*,\gamma}(0,T'):=\{f\in\mathcal H_*(0,T'):\ \textrm{supp}(f)\subset [0,T']\times\gamma\}.$$ Then, we consider the recovery of $F$ restricted to $[\delta,T)\times \gamma\times I$ from the data 
$$N_{F,\gamma}'(\lambda \chi)h,\quad h\in\mathcal H_{*,\gamma}(0,T'),\quad \lambda\in I,$$
with $I$ an interval of $\R$.
This result can be stated as follows.

\begin{thm}\label{t1} Let $n=2,3$, $F_1,F_2\in \mathcal A_*$ 
and fix $0<T< T_1(L+3\epsilon)$. Consider also $\delta>0$ and  $\chi\in\mathcal C^\infty_0((0,T'+1)\times \partial M)$ satisfying $\chi=1$ on $[\delta,T']\times\gamma$ and   $L_1:=\frac{L}{\norm{\chi}_{H^{\frac{5}{2}}((0,T')\times\partial M)}}$. Then the conditions
\bel{t1a} F_1(t,x,0)=F_2(t,x,0),\quad (t,x)\in[\delta,T]\times\gamma,\ee
\bel{t1b} N_{F_1,\gamma}'(\lambda \chi)h=N_{F_2,\gamma}'(\lambda \chi)h,\quad \lambda\in[-L_1,L_1],\ h\in\mathcal H_{*,\gamma}(0,T'),\ee
imply 
\bel{t1c} F_1(t,x,\lambda)=F_2(t,x,\lambda),\quad (t,x,\lambda)\in[\delta,T]\times\gamma\times[-L_1,L_1].\ee
\end{thm}

This first result corresponds to the recovery of the nonlinear term $F$ restricted to a portion $\gamma$ of the boundary of $M$. In order to recover $F$ inside $M$ we will first need additional information about $M$. Let us first recall the definition of simple manifolds.
\begin{Def}
\label{def_simple}
A compact smooth Riemannian manifold with boundary $(M,g)$ is simple if it is simply connected, the boundary $\pd M$ is strictly convex in the sense of the second fundamental form, and $M$ has no conjugate points. 
\end{Def}

With this additional assumption, we can extend Theorem \ref{t1} in the following way.
\begin{thm}\label{tt2} Let $n=2,3$, $M$ be a simple manifold, $F_1,F_2\in \mathcal A$ and fix $0<T< T_1(L+3\epsilon)$, $$L_2:=\frac{L}{\norm{1}_{H^{\frac{5}{2}}((0,T')\times\partial M)}+\norm{1}_{H^{\frac{5}{2}}( M)}},$$
with $1$ the constant function given by $[0,T]\times M\ni (t,x)\mapsto1$. Then the conditions 
\bel{tt2a} F_1(t,x,0)=F_2(t,x,0),\quad (t,x)\in(\{0\}\times M)\cup ((0,T)\times\partial M),\ee
\bel{tt2b} B_{F_1,\pd M}'(\lambda,\lambda,0)H=B_{F_2,\pd M}'(\lambda,\lambda,0)H,\quad \lambda\in\left[-L_2,L_2\right],\ H\in\mathcal H(0,T')\ee
imply
\bel{tt2c} F_1(t,x,\lambda)=F_2(t,x,\lambda),\quad (t,x,\lambda)\in[0,T]\times\partial M\times\left[-L_2,L_2\right],\ee
\bel{tt2d} F_1(0,x,\lambda)=F_2(0,x,\lambda),\quad (x,\lambda)\in  M\times\left[-L_2,L_2\right].\ee
Here $(\lambda,\lambda,0)$ denotes the element of $\mathcal H(0,T')$ corresponding to the different traces of the constant map $(t,x)\mapsto\lambda$.
\end{thm}

In the specific case of a bounded domain of $\R^n$, $n=2,3$, with Euclidean metric, we can give a more precise result with restriction of the data to some portion of the boundary and some restrictions of the solutions  at $t=0$. To state this result which will be our last main result, we consider first the following tools. For any $\omega\in\mathbb S^{n-1}=\{y\in\R^n:\ \abs{y}=1\}$ we consider the $\omega$-shadowed and $\omega$-illuminated faces of $\partial\Omega$
\[\partial\Omega_{+,\omega}=\{x\in\partial\Omega:\ \nu(x)\cdot\omega\geq0\},\quad \partial\Omega_{-,\omega}=\{x\in\partial\Omega:\ \nu(x)\cdot\omega\leq0\}.\]
 Here, for all $k\in\mathbb N^*$, $\cdot$ denotes the scalar product in $\R^k$ defined by
\[ x\cdot y=x_1y_1+\ldots +x_ky_k,\quad x=(x_1,\ldots,x_k)\in \R^k,\ y=(y_1,\ldots,y_k)\in \R^k.\]
We fix $\omega_0\in \mathbb S^{n-1}$ and we consider $U=[0,T']\times U'$ (resp $V=(0,T)\times V'$) with $U'$ (resp $V'$) an open  neighborhood of $\partial\Omega_{+,\omega_0}$ (resp $\partial\Omega_{-,\omega_0}$) in $\partial\Omega$. Let us also consider the following restriction of the space   $\mathcal H(0,T')$ given by
$$\mathcal H_U(0,T'):=\{H=(h,h_0,h_1)\in\mathcal H(0,T'):\ h_0=0,\ \textrm{supp}(h)\subset U\}.$$

\begin{thm}\label{tt3} Let $n=2,3$, $M=\overline{\Omega}$ with $\Omega$ an open connected and smooth domain of $\R^n$ with the Euclidean metric, let $F_1,F_2\in \mathcal A$ and fix $0<T<T_1(L+3\epsilon)$.  Then the conditions \eqref{tt2a} and
\bel{tt3a} B_{F_1,V}'(\lambda,\lambda,0)H=B_{F_2,V}'(\lambda,\lambda,0)H,\quad \lambda\in[-L_2,L_2],\ H\in\mathcal H_{U}(0,T')\ee
imply
\bel{tt3b} F_1(t,x,\lambda)=F_2(t,x,\lambda),\quad (t,x,\lambda)\in[0,T]\times\partial M\times\left[-L_2,L_2\right],\ee
\bel{tt3c} F_1(0,x,\lambda)=F_2(0,x,\lambda),\quad (x,\lambda)\in  M\times\left[-L_2,L_2\right].\ee

\end{thm}


\subsection{Comments about the main results}

To the best of our knowledge Theorem \ref{t1}, \ref{tt2} and \ref{tt3} are the first results of recovery of a general semilinear term appearing in a hyperbolic nonlinear equation from boundary measurements. Indeed, to the best of our knowledge one can only find results dealing with recovery of coefficients, appearing in a nonlinear hyperbolic equation, in the mathematical literature (see e.g. \cite{NM,NV}). It seems that such results have only been considered for parabolic or elliptic equations (e.g. \cite{CK,I3,I4,I5,IN,MU,SuU}).  Note also that like \cite{CK,I3,I4}, we manage to recover the nonlinear term at the lateral boundary $(0,T)\times\partial M$, with data restricted to the lateral boundary, but also inside the domain.

The proof of Theorem \ref{t1}, \ref{tt2} and \ref{tt3}, are based on a linearization procedure inspired by \cite{CK,I3,I4,I5}. The idea consists in transforming the recovery of the nonlinear term $F(t,x,u)$ into the recovery of time-dependent coefficients $q(t,x)=\partial_uF(t,x,u(t,x))$, where $u$ solves \eqref{eqn1} with suitable choice of the data $(f,u_0,u_1)$, appearing in a linear hyperbolic equation. So far this approach has been considered only with H\"older continuous solutions of some nonlinear parabolic or elliptic equations. For hyperbolic equations, the existence of such smooth solutions  seems to require at least strong assumptions on  the semilinear term under consideration. For this reason, in this paper, we provide, for what seems to be the first time, the extension of the linearization procedure considered  by \cite{I3}, to solutions lying in Sobolev space instead of H\"older continuous space. This extension of the analysis of \cite{I3} allows us to consider the case of nonlinear hyperbolic equations.

As mentioned above, our approach consists in transforming our inverse problem into the recovery of a time-dependent potential of the form $q(t,x)=\partial_uF(t,x,u(t,x))$,  where $u$ solves \eqref{eqn1}. This means that the regularity of the coefficient $q$ will depend explicitly on the solution of the nonlinear problem \eqref{eqn1}. For this reason, we can not apply results dealing with recovery of smooth time-dependent coefficients. In Theorem \ref{tt2} and \ref{tt3}, we use the results of \cite{HK,Ki1,Ki3,KO} dealing with the global recovery of such coefficients with low regularity assumptions. For Theorem \ref{t1}, we need to use results of recovery of time-dependent coefficients on the portion $(0,T)\times\gamma$ of  the lateral boundary $(0,T)\times\partial M$ from measurements restricted also to $(0,T)\times\gamma$. Moreover, we need to consider such results on some general Riemannian manifold. To the best of our knowledge \cite{SY} is the only work dealing with results close to the one needed for Theorem \ref{t1} (see also \cite{SU} for time-independent coefficients). However, the approach of \cite{SY}, based on local properties of general geometric optics solutions,  requires strong smoothness assumptions and it can not be applied in the context of Theorem \ref{t1}. For this reason we introduce a new approach for the recovery of less-regular coefficients in the proof of Theorem \ref{t3} (see Section 3). The result of Theorem \ref{t3} is based on a global construction of particular solutions of the linear problem \eqref{eq1}, with a control on their behavior close to the boundary. In contrast to other related results (e.g. \cite{SU,SY}) we do not restrict our analysis on some local properties of general geometric optics solutions associated with \eqref{eq1}, but some global construction in boundary normal coordinates suitably designed for any point $(t,x)\in (0,T)\times\gamma$.

In contrast to other related results for parabolic or elliptic equations (e.g. \cite{CK,I3,I4,I5}), we make only small restrictions on the class of nonlinear terms under consideration. Indeed, we even consider semilinear equations  with solutions that may blow-up at finite time. For this purpose, we state our result on, what can correspond to, the infimum of the final time  of existence, denoted by $T_1$, of maximal solutions associated with all possible  semilinear terms lying in $\mathcal A$. Here $T_1$ is a function of the size of the data $(f,u_0,u_1)$. We believe that with additional assumptions on the class of admissible nonlinear terms $\mathcal A$ (see \cite{BLP,Ka,IJ})  our result would be equivalent to the one stated by \cite{CK,I3,I4,I5} for global solutions of some nonlinear parabolic equations. However, in order to preserve the generality of our results, we prefer to keep this statement.

Let us observe, that, to the best of our knowledge, contrary to all other works dealing with recovery of nonlinear terms (e.g. \cite{CK,I3,I4,I5,IN,MU,SuU}), we do not state our results with the boundary map $\mathcal B_{F,\gamma_1}$ or $\mathcal N_{F,\gamma_1}$ associated with the nonlinear problem \eqref{eqn1}, but with some partial knowledge of their Fr\'echet derivative. By taking into account the important amount of data contained into $\mathcal B_{F,\gamma_1}$ or $\mathcal N_{F,\gamma_1}$, this statement of the main results  makes an important difference in terms of restriction of the data used for solving the inverse problem.

Our analysis is restricted to dimension of space $n=2,3$, but we believe that with suitable assumptions it could be extended to higher dimension. This restriction is due to the application of the Sobolev embedding theorem in the linearization procedure.

\subsection{Outline}

This paper is organized as follows. In Section 2,  we define the maps  $\mathcal B_{F,\gamma_1}$ and $\mathcal N_{F,\gamma_1}$ and we prove that they admit a Fr\'echet derivative associated with  solutions of linear wave equations with  time-dependent coefficients. In Section 3, we establish the recovery on the portion $(0,T)\times \gamma$ of a time-dependent potential from measurements of solutions of the linear problem restricted to $(0,T)\times \gamma$. We prove this result, which is stated in Theorem \ref{t3}, for coefficients $q\in H^2((0,T)\times M)\cap \mathcal C([0,T]\times M)$. In Section 4, we recall some results about recovery of time-dependent coefficients appearing in hyperbolic equations borrowed from \cite{Ki3,KO}. In Section 5, we combine all the arguments introduced in the preceding sections of the paper in order to complete the proof of Theorem \ref{t1}, \ref{tt2} and \ref{tt3}. Finally, in the Appendix we show  local existence of sufficiently smooth solutions of \eqref{eqn1}.

\section{Linearization of the inverse problem}
In this section we will  prove that the maps $\mathcal B_{F,\gamma_1}$ and $\mathcal N_{F,\gamma_1}$ are well defined and admit a continuous Fr\'echet derivative. 
According to Lemma \ref{l2} and \ref{ll2} (see the Appendix),  for all $L>0$ there exists $T_1(L)\in(0,T']$ such that for all $F\in\mathcal A$, $p>1$ satisfying \eqref{pp}, $T<T_1(L)$ and for all $(f,u_0,u_1)\in\mathcal H(0,T')$ satisfying 
$$\norm{f}_{H^{\frac{5}{2}}((0,T')\times\partial M)}+\norm{u_0}_{H^{\frac{5}{2}}( M)}+\norm{u_1}_{H^{\frac{3}{2}}( M)}\leq L,$$
the problem \eqref{eqn1} admits a unique solution   $u\in W^{1,\frac{p}{b-1}}(0,T;H^2(M))\cap W^{2,\frac{p}{b-1}}(0,T;H^1(M))$ satisfying \eqref{l2ab}. In the same way, applying Lemma \ref{ll2} we deduce that, for all $L>0$,  $F\in \mathcal A_*$ and for all  $f\in\mathcal H_*(0,T')$ satisfying 
$$\norm{f}_{H^{\frac{5}{2}}((0,T')\times\partial M)}\leq L,$$
   problem \eqref{eqn1}, with $u_0=u_1=0$, admits a unique solution lying in  $W^{1,\frac{p}{b-1}}(0,T;H^2(M))\cap W^{2,\frac{p}{b-1}}(0,T;H^1(M))$ satisfying \eqref{l2ab}. Using these results we can define the maps $\mathcal B_{F,\gamma_1}$ and $\mathcal N_{F,\gamma_1}$. We will now show that these maps admit a continuous Fr\'echet derivative that we will use for linearizing our inverse problem. For this purpose, we consider the following linear initial boundary value problem 

\bel{eqli1}
\left\{ \begin{array}{ll}  \pd_t^2 w-\Delta_{g} w + q w  =  0, & \mbox{in}\ (0,T)\times M,\\  
w  =  h, & \mbox{on}\ (0,T)\times \partial M,\\  
 w(0,\cdot)  =h_0,\quad  \pd_tw(0,\cdot)  =h_1 & \mbox{in}\ M,
\end{array} \right.
\ee
to which we associate the linear operator
$$\mathcal D_{q,\gamma_1}:\mathcal H(0,T') \ni H=(h,h_0,h_1)\longmapsto ({\partial_\nu w}_{|(0,T)\times\gamma_1},w(T,\cdot)_{|M})\in L^2((0,T)\times\gamma_1)\times H^1(M),$$
and for $w$ the solution of \eqref{eqli1}, with $h_0=h_1=0$, the linear operator
$$\Lambda_{q,\gamma_1}:\mathcal H_*(0,T') \ni h\longmapsto {\partial_\nu w}_{|(0,T)\times\gamma_1}\in L^2((0,T)\times\gamma_1).$$
From now on, for any $H=(h,h_0,h_1)\in\mathcal H(0,T')$, we denote by $\norm{H}_\mathcal H$ the norm defined by
$$\norm{H}_\mathcal H^2:=\norm{h}_{H^{\frac{11}{2}}((0,T')\times\partial M)}^2+\norm{u_0}_{H^{\frac{11}{2}}( M)}^2+\norm{u_1}_{H^{\frac{9}{2}}( M)}^2.$$
We proceed now to the following linearization of the maps $\mathcal B_{F,\gamma_1}$ and $\mathcal N_{F,\gamma_1}$ introduced in Section 1.1.

\begin{thm}\label{t2}  Assume that $n=2$ or $n=3$ and let $F\in \mathcal A$ $($resp. $F\in\mathcal A_*$$)$. Then, the maps $\mathcal B_{F,\gamma_1}$ $($resp. $\mathcal N_{F,\gamma_1}$$)$ admits a continuous Fr\'echet derivative $\mathcal B_{F,\gamma_1}'$ $($resp. $\mathcal N_{F,\gamma_1}'$$)$ on 
$$\{G\in\mathcal H(0,T'):\ \norm{G}_{H^{\frac{5}{2}}((0,T')\times\partial M)\times H^{\frac{5}{2}}( M)\times H^{\frac{3}{2}}( M)}\leq L\},$$
$$\left(\textrm{resp. }\{h\in\mathcal H_*(0,T'):\ \norm{h}_{H^{\frac{5}{2}}((0,T')\times\partial M)}\leq L\}\right).$$
Moreover, fixing  $$G\in\{K\in\mathcal H(0,T'):\ \norm{K}_{H^{\frac{5}{2}}((0,T')\times\partial M)\times H^{\frac{5}{2}}( M)\times H^{\frac{3}{2}}( M)}\leq L\},$$ 
$$ \left(\textrm{resp. }f\in\{h\in\mathcal H_*(0,T'):\ \norm{h}_{H^{\frac{5}{2}}((0,T')\times\partial M)}\leq L\}\right),$$
$q_{F,G}(t,x):=\partial_uF(t,x,u_{F,G}(t,x))$ (resp. $q_{F,f}(t,x):=\partial_uF(t,x,u_{F,f}(t,x))$), we find
\bel{t2a} \begin{aligned}&\mathcal B_{F,\gamma_1}'(G)H=\mathcal D_{q_{F,G},\gamma_1}H,\quad H\in\mathcal H(0,T')\\
&\left(\textrm{resp. }\mathcal N_{F,\gamma_1}'(f)h=\Lambda_{q_{F,f},\gamma_1}h,\quad  h\in\mathcal H_*(0,T')\right).\end{aligned}\ee

\end{thm}
\begin{proof} Since the proof for $\mathcal B_{F,\gamma_1}$ and $\mathcal N_{F,\gamma_1}$ are similar, we will only prove this result for $B_{F,\gamma_1}$. Moreover, without lost of generality, we assume that $\gamma_1=\partial M$.  For this purpose, we fix $H:=(h,h_0,h_1)\in \mathcal H(0,T')$ satisfying $\norm{H}_{H^{\frac{5}{2}}((0,T')\times\partial M)\times H^{\frac{5}{2}}( M)\times H^{\frac{3}{2}}( M)}+\norm{H}_\mathcal H\leq \epsilon$ and we consider $v=u_{F,G+H}-u_{F,G}-w$, with $w$ solving \eqref{eqli1} with $q=q_{F,G}$. By Taylor expansion in $u$ of $F$, we find
$$\begin{aligned}&F(t,x,u_{F,G+H}(t,x))\\
&=F(t,x,u_{F,G}(t,x))+\partial_uF(t,x,u_{F,G}(t,x))(u_{F,G+H}(t,x)-u_{F,G}(t,x))\\
\ &\ \ \ +\left(\int_0^1(1-s)\partial_u^2F(t,x,u_{F,G}(t,x)+s(u_{F,G+H}(t,x)-u_{F,G}(t,x)))ds\right)(u_{F,G+H}(t,x)-u_{F,G}(t,x))^2.\end{aligned}$$
Then, $v$ solves the linear problem
\bel{eqli2}
\left\{ \begin{array}{ll}  \pd_t^2 v-\Delta_{g}v + q_{F,G} v  =  -R, & \mbox{in}\ (0,T)\times M,\\  
v  =  0, & \mbox{on}\ (0,T)\times \partial M,\\  
v(0,\cdot)  =0,\quad  \pd_tv(0,\cdot)  =0 & \mbox{in}\ M,
\end{array} \right.
\ee
with
$$\begin{aligned}&R(t,x)\\
&:=\left(\int_0^1(1-s)\partial_u^2F(t,x,u_{F,G}(t,x)+s(u_{F,G+H}(t,x)-u_{F,G}(t,x)))ds\right)(u_{F,G+H}(t,x)-u_{F,G}(t,x))^2.\end{aligned}$$
Since $M$ is of dimension $n\leq3$, by the Sobolev embedding theorem, the space $W^{1,\frac{p}{b-1}}(0,T;H^2(M))$ embedded continuously into $\mathcal C([0,T]\times M)$ and we deduce that 
$$\norm{R}_{L^2(0,T;L^2( M))}\leq C\norm{R}_{L^\infty((0,T)\times M)}\leq C\norm{u_{F,G+H}-u_{F,G}}_{L^\infty((0,T)\times M)}^2.$$
Combining this with \cite[Theorem A.2]{BCY}, \cite[Proposition 1]{HK}, \eqref{l1ab} and applying the Sobolev embedding theorem, we obtain
\bel{t2b} \begin{aligned}\norm{\partial_\nu v}_{L^2((0,T)\times\partial M)}+\norm{v}_{\mathcal C([0,T];H^1(M))}&\leq C\left(\norm{R}_{L^1(0,T;L^2( M))}+\norm{q_{F,G}v}_{L^1(0,T;L^2( M))}\right)\\
\ &\leq C\left(\norm{R}_{L^2(0,T;L^2( M))}+\norm{q_{F,G}}_{L^1(0,T;L^3( M))}\norm{v}_{L^\infty(0,T;H^1( M))}\right)\\
\ &\leq C\norm{R}_{L^2(0,T;L^2( M))}\\
\ &\leq C\norm{u_{F,G+H}-u_{F,G}}_{L^\infty((0,T)\times M)}^2.\end{aligned}\ee
On the other hand, $y:=u_{F,G+H}-u_{F,G}$ solves the problem
\bel{eqli3}
\left\{ \begin{array}{ll}  \pd_t^2 y-\Delta_{g}y +V y=  0, & \mbox{in}\ (0,T)\times M,\\  
y  =  h, & \mbox{on}\ (0,T)\times \partial M,\\  
y(0,\cdot)  =h_0,\quad  \pd_ty(0,\cdot)  =h_1 & \mbox{in}\ M,
\end{array} \right.
\ee
with 
$$V(t,x):=\int_0^1\partial_uF(t,x,u_{F,G}(t,x) +s(u_{F,G+H}(t,x)-u_{F,G}(t,x)))ds.$$
Using the fact that $u_{F,G},\ u_{F,G+H}\in W^{1,\frac{b}{b-1}}(0,T;H^2(M))\subset W^{1,\frac{b}{b-1}}(0,T;L^\infty(M))$, we deduce that $V\in W^{1,\frac{b}{b-1}}(0,T;L^\infty(M))$. Thus, $y_1=\partial_ty$ solves
$$\left\{ \begin{array}{ll}  \pd_t^2 y_1-\Delta_{g}y_1 +V y_1=\partial_t V y  , & \mbox{in}\ (0,T)\times M,\\  
y_1  =  \partial_th, & \mbox{on}\ (0,T)\times \partial M,\\  
y_1(0,\cdot)  =h_1,\quad  \pd_ty_1(0,\cdot)  =\Delta_gh_0-V (0,\cdot)h_0 & \mbox{in}\ M,
\end{array} \right.$$
where one can check that
$$V (0,x)=\int_0^1\partial_uF(0,x,u_0(x)+sh_0(x))ds,\quad x\in M.$$
Combining this with the fact that  $G,H=(h,h_0,h_1)\in\mathcal H(0,T')$, we deduce from \cite[Proposition 1]{HK}  that this problem admits a unique solution $y_1\in\mathcal C([0,T];H^1(M))\cap \mathcal C^1([0,T];L^2(M))$, satisfying
$$\begin{aligned}\norm{y_1}_{\mathcal C^1([0,T];L^2(M))}&\leq C\left(\norm{H}_\mathcal H+\norm{\partial_t V y}_{L^{\frac{b}{b-1}}(0,T;L^2(M))}\right)\\
&\leq C\left(\norm{H}_\mathcal H+ \norm{V}_{W^{1,\frac{b}{b-1}}(0,T;L^\infty (M))}\norm{y}_{\mathcal C([0,T];L^2(M))}\right)\\
\ &\leq C\norm{H}_\mathcal H,\end{aligned}$$
with $C$ depending only on $c_1$, $T'$, $b$, $G$, $M$, $\epsilon$ and $T$. Note that here we use the fact that for $\norm{H}_\mathcal H\leq \epsilon $, $\norm{V}_{L^\infty((0,T)\times M)}$ and $\norm{V}_{W^{1,\frac{b}{b-1}}(0,T;L^\infty (M))}$ are upper bounded by a constant depending only on $T'$, $c_1$, $b$, $\epsilon$, $G$, $T$ and $M$. We apply also here the fact that the restriction operator $f\mapsto f_{|(0,T)\times\partial M}$ is a continuous map from $H^{\frac{11}{2}}((0,T')\times\partial M)$ to $H^{\frac{11}{2}}((0,T)\times\partial M)$. Thus, we have $y\in \mathcal C^2([0,T]; L^2(M))$ and
$$\norm{\Delta_g y}_{\mathcal C([0,T];L^2(M))}\leq \norm{\partial_t^2y}_{\mathcal C([0,T];L^2(M))}+\norm{V}_{L^\infty((0,T)\times M)}\norm{y}_{\mathcal C([0,T];L^2(M))}\leq C\norm{H}_\mathcal H.$$
Combining this with the fact that for all $t\in[0,T]$, $y(t,\cdot)$ solves the boundary value problem
$$\left\{ \begin{array}{ll}  -\Delta_{g}y(t,\cdot) =-\pd_t^2 y(t,\cdot)-V y(t,\cdot)  , & \mbox{in}\  M,\\  
y(t,\cdot)  =  h(t,\cdot), & \mbox{on}\  \partial M,
\end{array} \right.$$
we deduce that $y\in \mathcal C([0,T];H^2(M))$ satisfies the estimate
$$\norm{y}_{\mathcal C([0,T];H^2(M))}\leq C\norm{H}_{\mathcal H}.$$
Then, by the Sobolev embedding theorem, we obtain
$$\norm{u_{F,G+H}-u_{F,G}}_{L^\infty((0,T)\times M)}=\norm{y}_{L^\infty((0,T)\times M)}\leq C\norm{H}_{\mathcal H}$$
and, from \eqref{t2b}, we get 
$$\begin{aligned}&\norm{\partial_\nu u_{F,G+H}-\partial_\nu u_{F,G}-\partial_\nu w}_{L^2((0,T)\times\partial M)}+\norm{u_{F,G+H}-u_{F,G}-w}_{\mathcal C([0,T];H^1(M))}\\
&\leq C\norm{H}_{\mathcal H}^2.\end{aligned}$$
This proves that $\mathcal B_{F,\gamma_1}$ is Fr\'echet  differentiable at $G$ and $$\mathcal B_{F,\gamma_1}'(G)H=({\partial_\nu w}_{|(0,T)\times\gamma},w(T,\cdot)_{|M})=\mathcal D_{q_{F,G},\gamma_1}H.$$
Now let us prove the continuity of the map $G\mapsto \mathcal B_{F,\gamma_1}'(G)=\mathcal D_{q_{F,G},\gamma_1}$. For this purpose, we fix $z:=w_K-w$, with $K=(k,k_0,k_1)\in\mathcal H(0,T')$ and $w_K$ solving \eqref{eqli1} with $q=q_{F,G+K}$,
$$\norm{H}_{H^{\frac{5}{2}}((0,T)\times\partial M)\times H^{\frac{5}{2}}( M)\times H^{\frac{3}{2}}( M)}+\norm{K}_{H^{\frac{5}{2}}((0,T)\times\partial M)\times H^{\frac{5}{2}}( M)\times H^{\frac{3}{2}}( M)}+\norm{H}_\mathcal H+\norm{K}_\mathcal H\leq \epsilon.$$
We remark that $z$ solves the problem
\bel{eqli4}
\left\{ \begin{array}{ll}  \pd_t^2 z-\Delta_{g}z +q_{F,G}z=  S, & \mbox{in}\ (0,T)\times M,\\  
z  =  0, & \mbox{on}\ (0,T)\times \partial M,\\  
z(0,\cdot)  =0,\quad  \pd_tz(0,\cdot)  =0 & \mbox{in}\ M.
\end{array} \right.
\ee
with 
$$S=-(q_{F,G+K}-q_{F,G})w_K.$$
On the other hand, we can prove that $\norm{w_K}_{\mathcal C([0,T];H^2(M))}\leq C$ with $C$ depending on $a_1$, $b$, $G$, $\epsilon$, $M$, $T'$, $T$. Therefore, we find
\bel{t2d}\norm{S}_{L^2((0,T)\times M)}\leq C\norm{q_{F,G+K}-q_{F,G}}_{L^\infty((0,T)\times M)}.\ee
Using the Taylor expansion of $\partial_uF$ in $u$, we find
$$q_{F,G+K}(t,x)-q_{F,G}(t,x)=\left(\int_0^1\partial_u^2F(t,x,u_{F,G}+s(u_{F,G+K}-u_{F,G}))ds\right)(u_{F,G+K}-u_{F,G})$$
and repeating the above arguments, we obtain
$$\norm{q_{F,G+K}-q_{F,G}}_{L^\infty((0,T)\times M)}\leq C\norm{K}_\mathcal H.$$
Combining this with \eqref{t2d} and the estimate
$$\norm{\partial_\nu z}_{L^2((0,T)\times\partial M)}+\norm{ z}_{\mathcal C([0,T];H^1( M))}\leq C\norm{S}_{L^2((0,T)\times M)},$$
we deduce the continuity of $G\mapsto \mathcal B_{F,\gamma_1}'(G)=\mathcal D_{q_{F,G},\gamma_1}$. This completes the proof of the theorem.\end{proof}

\section{Recovery of a time-dependent coefficient on parts of the boundary}

For $T\in(0,T']$ and $q\in L^\infty((0,T)\times M)$ we consider the initial boundary value problem 
\bel{eq1}
\left\{ \begin{array}{ll}  \pd_t^2 u-\Delta_{g} u + q u  =  0, & \mbox{in}\ (0,T)\times M,\\  
u  =  f, & \mbox{on}\ (0,T)\times \partial M,\\  
 u(0,\cdot)  =0,\quad  \pd_tu(0,\cdot)  =0 & \mbox{in}\ M,
\end{array} \right.
\ee
with non-homogeneous Dirichlet data $f$. According to \cite{LLT}, for $f\in H^1((0,T)\times\partial M)$ satisfying $f_{|t=0}=0$ this problem admits a unique solution $u\in \mathcal C([0,T];H^1(M))\cap \mathcal C^1([0,T];L^2(M))$ satisfying $\partial_\nu u\in L^2((0,T)\times\partial M)$. Thus, fixing $\gamma$ an open portion of $\partial M$, we can define the partial hyperbolic Dirichlet-to-Neumann map in the following way
\[\Lambda_{q,\gamma,*}:\mathcal H_{*,\gamma}(0,T')\ni f\mapsto \partial_\nu u_{|(0,T)\times\gamma},\]
with $\mathcal H_{*,\gamma}(0,T'):=\{f\in \mathcal H_*(0,T'):\ \textrm{supp}(f)\subset(0,T']\times\gamma\}$ and with $u$ solving problem \eqref{eq1}. In this section, we consider the problem of recovering $q$ restricted to $(0,T)\times\gamma$ from the knowledge of $\Lambda_{q,\gamma,*}$.

\begin{thm}\label{t3} 
 Let $(M,g)$ be  a smooth connected and compact Riemannian manifold of dimension $n\geq2$ and let $q_1$, $q_2\in   \mathcal C([0,T]\times M)\cap H^2((0,T)\times M)$. 
Then $\Lambda_{q_1,\gamma,*} = \Lambda_{q_2, \gamma,*}$
implies that $q_1 = q_2$ on $(0,T)\times \gamma$.
\end{thm}

We mention that \cite{SU} established  results similar to Theorem \ref{t3} for time-independent coefficients and \cite{SY} treated the case of time-dependent coefficients from some measurements associated with some general hyperbolic equation on a Lorentzian manifold. Both of these results require strong smoothness assumptions on the coefficients under consideration. In Theorem \ref{t3}, we extend such results  to time-dependent potentials $q$ lying in $\mathcal C([0,T]\times M)\cap H^2((0,T)\times M)$. To prove this result, like in \cite{SU,SY}, we consider specific solutions of the problem \eqref{eq1} also called geometric optics. However, since we restrict the regularity of the coefficients under consideration, in contrast to \cite{SU,SY},  we will use a new global construction involving some  approximation of the potential $q$. The properties of these solutions will be stated in Proposition \ref{p2}. We mention that the recovery of coefficients lying in  $\mathcal C([0,T]\times M)\cap H^2((0,T)\times M)$ will be a crucial point in the proof of Theorem \ref{t1}.

\subsection{Geometric optics solutions }
Let $t_0\in (0,T)$, $x_0\in\partial M$ and consider $\delta>0$ a constant that will be fixed later. The goal of this subsection is to construct some energy class solutions $u_j$ of the equation
\bel{eq3}\left\{ \begin{array}{ll}  \pd_t^2 u_{j}-\Delta_{g} u_j+ q_j u_j  =  0, & \mbox{in}\ (0,T)\times M,\\  
u_j  =  f, & \mbox{on}\ (0,T)\times \partial M,\\  
 u_j(0,\cdot)  =0,\quad  u_j(0,\cdot)  =0 & \mbox{in}\ M,
\end{array} \right.
\ee
with some suitable choice of $f\in\mathcal C^\infty_0((0,T]\times \pd M)$. More precisely, we prove the following.

\begin{prop}\label{p2} 
 For $j=1,2$ and for  $\rho>1$, there exists $f\in\mathcal C^3_0((0,T]\times \pd M)$ such that the solution $u_j\in \mathcal C([0,T];H^1(M))\cap \mathcal C^1([0,T];L^2(M))$ of \eqref{eq3} has a restriction on $[0,t_0+\delta]\times  M$ taking the form
\bel{GO} u_j(t,x)= e^{i\rho(t-\psi(x))} \left(a_0(t,x)+\frac{a_{j,1}(t,x)}{\rho}+\frac{a_{j,2,\rho}(t,x)}{\rho^2}\right)+R_{j,\rho}(t,x), \quad \rho > 1.
\ee
Here we assume that $\psi$ is a smooth function on the support of $a_0$, $a_{j,1}$, $a_{j,2,\rho}$.  Moreover, the function $a_0$, $a_{j,1}$, $a_{j,2,\rho}\in H^2((0,t_0+\delta)\times  M)$ satisfy the conditions
\bel{bound}a_{1,1}(t,x)=a_{2,1}(t,x)=0,\quad (t,x)\in(0,t_0+\delta)\times\partial M,\ee
\bel{bound1}a_{1,2,\rho}(t,x)=a_{2,2,\rho}(t,x)=0,\quad (t,x)\in(0,t_0+\delta)\times\partial M,\ee
\bel{supp}\textrm{supp}\left(a_0|_{[0,t_0+\delta]\times  \pd M}\right)\subset [0,t_0+\delta]\times  \gamma,\ee
\bel{bounded1} \norm{a_{j,2,\rho}}_{H^2((0,t_0+\delta)\times M)}\leq C\rho^{\frac{n}{n+2}},\ee
with $C>0$ independent of $\rho$.
In addition, we have $\partial_\nu a_{j,1}\in \mathcal C([0,t_0+\delta]\times  \partial M)$, $j=1,2$, with
\bel{bounded}\partial_\nu a_{1,1}(t_0,x_0)-\partial_\nu a_{2,1}(t_0,x_0)=\frac{i}{2}[q_1(t_0,x_0)-q_2(t_0,x_0)].\ee
Finally, the remainder term $R_{j,\rho}\in \mathcal C([0,t_0+\delta];H^1(M))\cap \mathcal C^1([0,t_0+\delta];L^2(M))$ satisfies
$$\partial_t^2R_{j,\rho}-\Delta_gR_{j,\rho}\in L^2((0,t_0+\delta)\times M),$$
\begin{align}
\label{GO1}
&R_{j,\rho}=0 \textrm{ on }(0,t_0+\delta)\times\pd M,\quad R_{j,\rho}(0,\cdot)=\pd_tR_{j,\rho}(0,\cdot)=0 \textrm{ on } M,
\\\label{GO2}
&\lim_{\rho\to+\infty}  \rho\norm{\partial_\nu R_{j,\rho}}_{L^2((0,t_0+\delta)\times \partial M)}=0.
\end{align}
\end{prop}
\begin{proof} 
In order to get the decay \eqref{GO2}, we choose $\psi$, $a_0$, $a_{j,1}$ and $a_{j,2,\rho}$, $j=1,2$, so that they satisfy the following eikonal and transport equations
\bel{psi}\sum_{i,j=1}^dg^{ij}(x)\pd_{x_i}\psi\pd_{x_j}\psi=|\nabla_g\psi|^2_g=1,\ee
\bel{a0}  2i\pd_ta_0+2i\sum_{i,j=1}^dg^{ij}(x)\pd_{x_i}\psi\pd_{x_j}a_0+i(\Delta_g \psi)a_0=0,\ee
\bel{a1}  2i\pd_ta_{k,1}+2i\sum_{i,j=1}^dg^{ij}(x)\pd_{x_i}\psi\pd_{x_j}a_{k,1}+i(\Delta_g \psi)a_{k,1}=-(\partial_t^2-\Delta_g+q_k)a_0,\quad k=1,2,\ee
\bel{a2}  2i\pd_ta_{k,2,\rho}+2i\sum_{i,j=1}^dg^{ij}(x)\pd_{x_i}\psi\pd_{x_j}a_{k,2,\rho}+i(\Delta_g \psi)a_{k,2,\rho}=-(\partial_t^2-\Delta_g+q_k)a_{k,1, \rho},\quad k=1,2,\ee
on some neighborhood of $[0,t_0+\delta]\times \partial M$. Here $a_{k,1, \rho}$ is a smooth approximation of $a_{k,1}$ that we will precise later. Using some suitable coordinates we will introduce solutions of the equations \eqref{psi}-\eqref{a2} satisfying \eqref{bound}-\eqref{bounded1}. 

From now on, for any $y\in M$ and $\theta\in S_yM$, we denote by  $\gamma_{y,\theta}$ the maximal geodesic starting at $y$ in the direction $\theta$. Then, for some $\epsilon>0$ small enough, we define the map $\exp_{\partial M}:\partial M\times[0,\epsilon)\longrightarrow M$ given by
$$\exp_{\partial M}(x',x_n):=\gamma_{x',-\nu(x')}(x_n),\quad (x',x_n)\in \partial M\times[0,\epsilon).$$
For any $r>0$, we define the submanifold $M_r:=\{x\in M:\ \textrm{dist}(x,\partial M)< r\}$. It is well known (e.g. \cite[Section 2.1.16]{KKL}) that, for $\epsilon$ sufficiently small, $\exp_{\partial M}$ is a diffeomorphism from 
$\partial M\times[0,\epsilon)$ to $M_\epsilon$ with $$\exp_{\partial M}^{-1}(x):=(x',x_n),\quad x_n=\textrm{dist}(x,\partial M),\quad x\in M_\epsilon.$$
Here $\textrm{dist}$ denotes the Riemanian distance function on $(M,g)$.
Thus,  we can consider the boundary normal coordinates  $(x',x_n)$ on $M_\epsilon$ given by $x=\exp_{\partial M}(x',x_n)$ where $x_n\geq0$ and $x'\in\partial M$.  It is well known  (see e.g. \cite[Section 2.1.18]{KKL}) that in these coordinates the metric takes the form $g(x',x_n)=g_0(x',x_n)+dx_n^2$ with $g_0(x',x_n)$ a metric on $\partial M$ that depends smoothly on $x_n$. We choose 
\begin{align}
\label{psi1}
\psi(x)=\textrm{dist}(x,\partial M), \quad x \in M_\epsilon.
\end{align}
As $\psi$ is given by $x_n$ in the boundary normal coordinates,  one can easily check that $\psi$ solves \eqref{psi} in $M_\epsilon$. 

\
Let us now turn to the transport equations \eqref{a0}-\eqref{a2}. We fix $\delta\in\left(0,\frac{\min(\epsilon,t_0,T-t_0)}{16}\right)$. From now on, we use the coordinates $s_1=t+x_n-t_0$, $s_2=t-x_n-t_0$ and, for $s_1\in [-t_0,5\delta],\ s_2\in[s_1-8\delta,s_1],\ x'\in\partial M$, we   write $$a(s_1,s_2,x')=a\left(\frac{s_1+s_2}{2},\exp_{\partial M}\left(x',\frac{s_1-s_2}{2}\right)\right).$$ 
We will use this notation to indicate the representation in these coordinates also for other functions. Note that in these coordinates the boundary $\partial M$ will be given by $s_1=s_2$ which corresponds in boundary normal coordinates to $x_n=0$. Moreover, the manifold $[0,t_0+\delta]\times M$ will be contained into the set \bel{coM}\{ (s_1,s_2,x')\in\R\times\R\times\pd M:\ s_2\leq s_1,\ -t_0\leq s_1+s_2\leq \delta\}.\ee
Fixing $\beta=\textrm{det}g_0$, one can check that, in the coordinates $(s_1,s_2,x')$, \eqref{a0} becomes
\[ 
2\pd_{s_1}a_0+\left({\pd_{s_1}\beta -\pd_{s_2}\beta\over 4\beta}\right)a_0=0.\] 
Then,  we consider $\chi\in\mathcal C^\infty_0((-2\delta,2\delta))$  such that $\chi=1$ on $[-\delta,\delta]$,  $\chi_1\in\mathcal C^\infty_0((-3\delta,3\delta))$  such that $\chi_1=1$ on $[-2\delta,2\delta]$,   $\phi\in\mathcal C^\infty_0(\gamma)$ such that $\phi=1$ on a neighborhood of $x_0$ and $\phi_1\in\mathcal C^\infty_0(\gamma)$ such that $\phi_1=1$ on a neighborhood of supp$(\phi)$. We choose 
\bel{a00}a_0(s_1,s_2,x'):=\chi(s_2)\phi(x')\beta(s_1,s_2,x')^{-1/4}.\ee
Using the fact that 
$$a_0(s_1,s_2,x')=0,\quad |s_2|>2\delta,\ s_1\in[s_2,3\delta],\ x'\in\partial M$$
we can extend $a_0$ by zero to a function defined on $s_1\in[-t_0-2\delta,3\delta]$, $s_2\in [-t_0-3\delta, s_1]$, $x'\in\pd M$ solving \eqref{a0} on $(0,t_0+\delta)\times M$. Then using the fact that $[0,t_0+\delta]\times M$ is described by \eqref{coM}, we deduce that this extension of $a_0$ corresponds to a function defined on $[0,t_0+\delta]\times M$  in the initial coordinates and lying in $\mathcal C^\infty([0,t_0+\delta]\times M)$.
With this choice of $a_0$, \eqref{a1} is transformed into 
\[ 2\pd_{s_1}a_{j,1}+\left({\pd_{s_1}\beta -\pd_{s_2}\beta\over 4\beta}\right)a_{j,1}=\frac{i}{2}[(\partial_t^2-\Delta_g)a_0+ q_ja_0].\]
We choose
\bel{a11}a_{j,1}(s_1,s_2,x_n):= \frac{i}{2}\chi_1(s_2)\phi_1(x')\beta(s_1,s_2,x')^{-\frac{1}{4}}\left(a_{3}(s_1,s_2,x_n)+ a_{j,3}(s_1,s_2,x')\right),\ee
on $(0,t_0+\delta)\times M_{2\delta}$, where, for $s_1\in [-t_0,5\delta],\ s_2\in[-3\delta,3\delta],\ x'\in\partial M$, we fix
$$a_{j,3}(s_1,s_2,x'):=\chi(s_2)\phi(x')\frac{1}{2}\left(\int_{s_2}^{s_1}\beta^{\frac{1}{4}}q_j(\tau,s_2,x')d\tau\right),\quad a_{3}(s_1,s_2,x')=\frac{1}{2}\left(\int_{s_2}^{s_1}\beta^{\frac{1}{4}}d_1(\tau,s_2,x')d\tau\right),$$
with $d_1=(\partial_t^2-\Delta_g)a_0$.
It is clear that 
$$a_{3}(s_1,s_1,x')= a_{j,3}(s_1,s_1,x')=0.$$
Thus, one can check that \eqref{bound} is fulfilled. Moreover, using the fact that $q\in H^2((0,T)\times M)\cap \mathcal C([0,T]\times M)$, we deduce  that  $a_{j,1}\in H^2((0,t_0+\delta)\times M)$ and $\partial_\nu a_{j,1}\in \mathcal C([0,t_0+\delta]\times \partial M)$, $j=1,2$.
Finally, due to the expression involving $\phi$ in \eqref{a00} one can check \eqref{supp}. Finally, using the fact that 
$$\begin{aligned} &\left(\partial_\nu a_{1,1}-\partial_\nu a_{2,1}\right)(t_0,x_0)\\
&=(\partial_{s_1}-\partial_{s_2})\left[\beta(s_1,s_2,x')^{-\frac{1}{4}}\chi(s_2)\phi(x_0)\frac{i}{4}\left(\int_{s_2}^{s_1}\beta^{\frac{1}{4}}(q_1-q_2)(\tau,s_2,x_0)d\tau\right)\right]_{s_2=s_1=0,}\\
&=\frac{i}{2}\chi(0)\phi(x_0)(q_1-q_2)(0,0,x_0)=\frac{i}{2}(q_1-q_2)(t_0,x_0),\end{aligned}$$
we obtain \eqref{bounded}.

For the construction of $a_{j,2,\rho}$, we need first to define the expression $a_{j,1,\rho}$ which is an approximation of $a_{j,1}$. For this purpose, we consider an approximation of $q_j$ given by the following lemma.
\begin{lem} \label{appro}There exists $q_{j,\rho}\in\mathcal C^\infty([0,T]\times M)$ such that 
\bel{app1} \lim_{\rho\to+\infty} \norm{q_{j,\rho}-q_j}_{H^2((0,T)\times M)}=0,\ee
\bel{app2}\norm{q_{j,\rho}}_{H^{\ell}(\R\times M_1)}\leq C_\ell\rho^{\frac{\ell-2}{n+2}},\ \ell\geq 2.\ee
with $C_\ell$  independent of $\rho$.\end{lem}

We postpone the proof of this result to the end of the present demonstration. Using the result of Lemma \ref{appro}, we fix
$$a_{j,3,\rho}(s_1,s_2,x'):=\chi(s_2)\phi(x')\frac{1}{2}\left(\int_{s_2}^{s_1}\beta^{\frac{1}{4}}q_{j,\rho}(\tau,s_2,x')d\tau\right)$$
 for $s_1\in [-t_0,5\delta],\ s_2\in[-t_0-5\delta,s_1],\ x'\in\partial M$
and we define $a_{j,1,\rho}$ as follows
$$a_{j,1,\rho}(s_1,s_2,x'):= \frac{i}{2}\chi_1(s_2)\phi_1(x')\beta(s_1,s_2,x')^{-\frac{1}{4}}\left(a_{3}(s_1,s_2,x')+ a_{j,3,\rho}(s_1,s_2,x')\right).$$
Then, according to \eqref{app1}-\eqref{app2} and the expression \eqref{a11} of $a_{j,1}$, we have $a_{j,1,\rho}\in\mathcal C^\infty([0,t_0+\delta]\times M)$ with 
\bel{app3} \lim_{\rho\to+\infty} \norm{a_{j,1,\rho}-a_{j,1}}_{H^2((0,t_0+\delta)\times M)}=0,\ee
\bel{app4}\norm{a_{j,1,\rho}}_{H^{\ell}((0,t_0+\delta)\times M_1)}\leq C_\ell\rho^{\frac{\ell-2}{n+2}},\ \ell\geq 2.\ee
Note that in the coordinates $(s_1,s_2,x')$, \eqref{a2} becomes
\bel{a221}
2\pd_{s_1}a_{j,2,\rho}+\left({\pd_{s_1}\beta -\pd_{s_2}\beta\over 4\beta}\right)a_{j,2,\rho}=\frac{i}{2}[(\partial_t^2-\Delta_g)a_{j,1,\rho}+ q_ja_{j,1,\rho}].\ee
Thus, for $s_1\in [-t_0,5\delta],\ s_2\in[-t_0-5\delta,s_1],\ x'\in\partial M$, we fix
\bel{a22}a_{j,2,\rho}(s_1,s_2,x'):= \chi_1(s_2)\phi_1(x')\beta(s_1,s_2,x')^{-\frac{1}{4}}\frac{1}{4i}\left(\int_{s_2}^{s_1}\beta^{\frac{1}{4}}b_{j,1,\rho}(\tau,s_2,x')d\tau\right),\ee
where
$$b_{j,1,\rho}:=-(\partial_t^2-\Delta_g+q_j)a_{j,1,\rho}.$$
In particular, we have $a_{j,2,\rho}\in\mathcal C^\infty([0,t_0+\delta]\times M)$ and \eqref{bound1}.

Combining these properties with the fact that $[0,t_0+\delta]\times M$ is contained into the set \eqref{coM}, we can extend the map
$$G_{j,\rho}:(t,x)\longmapsto e^{i\rho(t-\psi(x))} \left(a_0(t,x)+\frac{a_{j,1}(t,x)}{\rho}+\frac{a_{j,2,\rho}(t,x)}{\rho^2}\right)$$
by zero to a function lying in $H^2((0,t_0+\delta)\times M)$. In addition, \eqref{app3}-\eqref{app4} imply that 
$$ \begin{aligned}&\norm{\partial_t^2G_{j,\rho}-\Delta_gG_{j,\rho}+q_jG_{j,\rho}}_{L^2((0,t_0+\delta)\times M)}\\
&=\norm{\frac{(\partial_t^2-\Delta_g+q_j)(a_{j,1}-a_{j,1,\rho})}{\rho}+\frac{\partial_t^2a_{j,2,\rho}-\Delta_ga_{j,2,\rho}+q_ja_{j,2,\rho}}{\rho^2}}_{L^2((0,t_0+\delta)\times M)}\\
&\leq C\left(\rho^{-1}\norm{a_{j,1}-a_{j,1,\rho}}_{H^2((0,t_0+\delta)\times M)}+\frac{\norm{a_{j,2,\rho}}_{H^2((0,t_0+\delta)\times M)}}{\rho^2}\right).\end{aligned}$$
On the other hand, by considering the explicit expression of $a_{j,2,\rho}$ and applying \eqref{app4}, we get
$$ \begin{aligned}\norm{a_{j,2,\rho}}_{H^2((0,t_0+\delta)\times M)}&\leq C(\norm{a_{j,1,\rho}}_{H^4((0,t_0+\delta)\times M)}+\norm{q_ja_{j,1,\rho}}_{H^2((0,t_0+\delta)\times M)})\\
\ &\leq C(\norm{a_{j,1,\rho}}_{H^4((0,t_0+\delta)\times M)}+\norm{q_j}_{H^2((0,t_0+\delta)\times M)}\norm{a_{j,1,\rho}}_{W^{2,\infty}((0,t_0+\delta)\times M)})\\
\ &\leq C(\norm{a_{j,1,\rho}}_{H^4((0,t_0+\delta)\times M)}+\norm{q_j}_{H^2((0,t_0+\delta)\times M)}\norm{a_{j,1,\rho}}_{H^{n+2}((0,t_0+\delta)\times M)})\\
\ &\leq C\rho^{\frac{n}{n+2}},\end{aligned}$$
which implies \eqref{bounded1}.
Therefore, we find
$$ \begin{aligned}&\norm{\partial_t^2G_{j,\rho}-\Delta_gG_{j,\rho}+q_jG_{j,\rho}}_{L^2((0,t_0+\delta)\times M)}\\
&\leq C\left(\rho^{-1}\norm{a_{j,1}-a_{j,1,\rho}}_{H^2((0,t_0+\delta)\times M)}+\rho^{-\frac{n+4}{n+2}}\right).\end{aligned}$$
and  \eqref{app3} implies 
\bel{dec}\lim_{\rho\to+\infty}\rho \norm{\partial_t^2G_{j,\rho}-\Delta_gG_{j,\rho}+q_jG_{j,\rho}}_{L^2((0,t_0+\delta)\times M)}=0.\ee
We choose $R_{j,\rho}\in \mathcal C([0,t_0+\delta];H^1(M))\cap \mathcal C^1([0,t_0+\delta];L^2(M))$ to be the unique solution of the IBVP
\bel{eq2_R}
\left\{ \begin{array}{ll}  \pd_t^2 R_{j,\rho}-\Delta_{g} R_{j,\rho}+ q_j R_{j,\rho}  =  -(\partial_t^2G_{j,\rho}-\Delta_gG_{j,\rho}+q_jG_{j,\rho}), & \mbox{in}\ (0,t_0+\delta)\times M,\\  
R_{j,\rho}  =  0, & \mbox{on}\ (0,t_0+\delta)\times \partial M,\\  
 R_{j,\rho}(0,\cdot)  =0,\quad  \pd_tR_{j,\rho}(0,\cdot)  =0 & \mbox{in}\ M.
\end{array} \right.
\ee
Applying \cite[Theorem 2.1]{LLT}, we obtain

$$\begin{aligned}\norm{\partial_\nu R_{j,\rho}}_{L^2((0,t_0+\delta)\times \partial M)}&\leq C\left(\norm{\partial_t^2G_{j,\rho}-\Delta_gG_{j,\rho}+q_jG_{j,\rho}}_{L^2((0,t_0+\delta)\times M)}+\norm{ q_jR_{j,\rho}}_{L^2((0,t_0+\delta)\times M)}\right)\\
\ &\leq C\left(\norm{\partial_t^2G_{j,\rho}-\Delta_gG_{j,\rho}+q_jG_{j,\rho}}_{L^2((0,t_0+\delta)\times M)}+\norm{ R_{j,\rho}}_{\mathcal C([0,t_0+\delta]; H^1( M))}\right)\\
\ &\leq C\norm{\partial_t^2G_{j,\rho}-\Delta_gG_{j,\rho}+q_jG_{j,\rho}}_{L^2((0,t_0+\delta)\times M)}\end{aligned}$$
and \eqref{dec} implies \eqref{GO2}.

Using the above properties, we can now complete the construction of the solutions $u_j$ of \eqref{eq3}. 
Note first that, according to \eqref{bound} and \eqref{bound1}, we have
$$G_{1,\rho}(t,x)=G_{2,\rho}(t,x)=a_0(t,x):=f(t,x),\quad (t,x)\in[0,t_0+\delta]\times\pd M.$$
Using the fact that $f\in \mathcal C^\infty([0,t_0+\delta]\times \partial M)$, satisfies $f_{|(0,t_0-2\delta)\times \partial M}=0$,  we extend $f$ by symmetry in $t$ to an element of $\mathcal C^3_0((0,+\infty)\times \partial M)$.  Then, we fix 
 $u_j$, $j=1,2$, respectively the solution of the initial boundary value problem \eqref{eq3}.
Since the restriction of $u_j$ to $(0,t_0+\delta)\times M$ solves the initial boundary value   problem
$$\left\{ \begin{array}{ll}  \pd_t^2 u_{j}-\Delta_{g} u_j+ q_j u_j  =  0, & \mbox{in}\ (0,t_0+\delta)\times M,\\  
u_j  =  G_{j,\rho}, & \mbox{on}\ (0,t_0+\delta)\times \partial M,\\  
 u_j(0,\cdot)  =0,\quad  u_j(0,\cdot)  =0 & \mbox{in}\ M.
\end{array} \right.$$
by the uniqueness of the solution of this    problem we deduce that $u_j$ takes the form \eqref{GO}  on $(0,t_0+\delta)\times M$.

\end{proof}

Now that we have completed Proposition \ref{p2}, let us show Lemma \ref{appro}

\textbf{Proof of Lemma \ref{appro}.} We consider first $(M_j,g)$, $j=1,2$,  two compact an smooth connected  manifolds such that $M$ is contained into Int$(M_1)$, $M_1$ is contained into Int$(M_2)$. Then, we fix $q_{j*}\in  H^2(\R\times M_1)$ supported on $(-1,T+1)\times Int(M_1)$, which coincides with $q_j$ on $(0,T)\times M$ such that
$$\norm{q_{j*}}_{H^2(\R\times M_1)}\leq C\norm{q_j}_{H^2((0,T)\times M)},$$
with $C>0$ depending only on $M_1$, $T$. We fix the following local coordinates in $M_2$: 
$$(\phi_1,U_1),\ldots, (\phi_m,U_m)$$
such that $$M_1\subset\bigcup_{k=1}^n U_k\subset \textrm{Int}(M_2).$$
We fix also $\psi_k\in\mathcal C^\infty_0(U_k)$, $k=1,\ldots,m$, such that
$$\sum_{k=1}^n\psi_k(x)=1,\quad x\in M_1$$
and $\psi_{k,\sharp}\in \mathcal C^\infty_0(U_k)$, $k=1,\ldots,m$,  satisfying $\psi_{k,\sharp}=1$ on supp$(\psi_k)$.
Then, we set $\zeta\in\mathcal C^\infty_0(\R^{1+n})$ such that supp$(\zeta)\subset \{(t,x):\ |(t,x)|\leq 1\}$, $\zeta\geq0$ and
$$\int_{\R^{1+n}}\zeta(t,x)dxdt=1.$$
We consider also $\zeta_\rho(t,x)=\rho^{\frac{n+1}{n+2}}\zeta(\rho^{\frac{1}{n+2}}t, \rho^{\frac{1}{n+2}}x)$ and, for $j=1,2$ and $k=1,\ldots,m$, we define
$$\begin{aligned}q_{j,k,\rho}(t,y)&=\zeta_\rho*((\phi_k^{-1})^*\psi_{k,\sharp}q_{j*})(t,y)\\
\ &=\int_{\R^{1+n}}\zeta_\rho(t-s,y-z)\psi_{k,\sharp}(\phi_k^{-1}(z))\tilde{q}_j(s,(\phi_k^{-1}(z))dsdz,\quad j=1,2,\ (t,y)\in\R^{1+n},\end{aligned}$$
and we consider
$$q_{j,\rho}(t,x)=\sum_{k=1}^mq_{j,k,\rho}(t,\phi_k(x))\psi_k(x),\quad (t,x)\in\R\times M_1,\ j=1,2.$$
Note that
$$\begin{aligned}\norm{q_{j,\rho}-q_j}_{L^2((0,T)\times M)}&=\norm{\sum_{k=1}^m(\phi_k^*q_{j,k,\rho}-q_j\psi_{k,\sharp})\psi_k}_{L^2((0,T)\times M)}\\
\ &\leq C\sum_{k=1}^m \norm{\phi_k^*q_{j,k,\rho}-q_j\psi_{k,\sharp}}_{L^2((0,T)\times U_k)}\\
\ &\leq C\sum_{k=1}^m \norm{q_{j,k,\rho}-(\phi_k^{-1})^*q_j\psi_{k,\sharp}}_{L^2((0,T)\times \phi_k(U_k))}\\
\ &\leq C\sum_{k=1}^m \norm{q_{j,k,\rho}-(\phi_k^{-1})^*q_{j*}\psi_{k,\sharp}}_{L^2(\R^{1+n})},\end{aligned}$$
with $C$ depending only on $(M,g)$, $U_k$ and $\phi_k$, $k=1,\ldots,m$. Combining this with the fact that 
$$\limsup_{\rho\to+\infty}\norm{q_{j,k,\rho}-(\phi_k^{-1})^*q_{j*}\psi_{k,\sharp}}_{L^2(\R^{1+n})}=\limsup_{\rho\to+\infty}\norm{\zeta_\rho*((\phi_k^{-1})^*\psi_{k,\sharp}q_{j*})-(\phi_k^{-1})^*q_{j*}\psi_{k,\sharp}}_{L^2(\R^{1+n})}=0$$
we deduce that 
$$\lim_{\rho\to+\infty}\norm{q_{j,\rho}-q_j}_{L^2((0,T)\times M)}=0.$$
In the same way, using the fact that $q_{j*}\in  H^2(\R\times M_1)$, we deduce \eqref{app1}-\eqref{app2}.\qed

Applying Proposition \ref{p2}, we are now in position to complete the proof of Theorem \ref{t3}.

\subsection{Proof of Theorem \ref{t3}.}
In this subsection we consider solutions $u_j\in \mathcal C([0,T];H^1(M))\cap \mathcal C^1([0,T];L^2(M))$ of \eqref{eq3} given by Proposition \ref{p2}. Note that following the proof of Proposition \ref{p2}, thanks to \eqref{supp}, we know that supp$\left(u_j|_{(0,T)\times\pd M}\right)\subset (0,T)\times\gamma$. Therefore, the condition $\Lambda_{q_1,\gamma,*}=\Lambda_{q_2,\gamma,*}$ implies 
$$(\partial_\nu u_1-\partial_\nu u_2)(t,x)=0,\quad (t,x)\in(0,T)\times\gamma.$$
On the other hand, applying \eqref{bound} and \eqref{bound1}, for all $(t,x)\in (0,t_0+\delta)\times\gamma$, we obtain
\bel{lele}\begin{aligned}0&=\rho(\partial_\nu u_1-\partial_\nu u_2)(t,x)\\
&= e^{i\rho(t-\psi(x))}(\partial_{\nu} a_{1,1}-\partial_{\nu} a_{2,1})(t,x)+\frac{e^{i\rho(t-\psi(x))}(\partial_{\nu} a_{1,2,\rho}-\partial_{\nu} a_{2,2,\rho})(t,x)}{\rho}\\
&\ \ \ +\rho(\partial_{\nu} R_{1,\rho}-\partial_{\nu} R_{2,\rho})(t,x).\end{aligned}\ee
Applying \eqref{bounded1}, we find
$$\begin{aligned}\norm{\partial_{\nu} a_{1,2,\rho}-\partial_{\nu} a_{2,2,\rho}}_{L^2((0,t_0+\delta)\times\partial M)}&\leq C(\norm{a_{1,2,\rho}}_{H^2((0,t_0+\delta)\times M)}+\norm{a_{2,2,\rho}}_{H^2((0,t_0+\delta)\times M)})\\
\ &\leq C\rho^{\frac{n}{n+2}}.\end{aligned}$$
Combining this with \eqref{GO2} and sending $\rho\to+\infty$ in \eqref{lele}, we obtain
$$\norm{\partial_{\nu} a_{1,1}-\partial_{\nu} a_{2,1}}_{L^2((0,t_0+\delta)\times\gamma)}=0.$$
It follows that
$$\partial_{\nu} a_{1,1}(t,x)-\partial_{\nu} a_{2,1}(t,x)=0,\quad (t,x)\in(0,t_0+\delta)\times\gamma.$$
Combining this with \eqref{bounded}, we deduce  that $q_1(t_0,x_0)=q_2(t_0,x_0)$. Due to the arbitrary choice for $t_0\in(0,T)$ and $x_0\in\gamma$, this equality completes the proof of Theorem \ref{t3}.
\section{Recovery of time-dependent coefficients inside the domain}

In this section we will recall some results related to the recovery of time-dependent coefficients $q$, appearing in the linear problem \eqref{eqli1}, inside the manifold $M$. Our first result is stated on a simple manifold and it concerns recovery of time-dependent coefficients inside the manifold with restriction of the data on the  top $t=T$ of the time-space manifold $(0,T)\times M$. 

\begin{thm}\label{tin1} 
Assume that $(M,g)$ is a simple manifold. Let $T > 0$ and 
let $q_1$, $q_2\in  L^\infty((0,T)\times M)$. 
Then the condition $$\mathcal D_{q_1,\partial M}H=\mathcal D_{q_2,\partial M}H,\quad H\in\mathcal H(0,T')$$ implies that $q_1 = q_2$.
\end{thm}
This result follows from  \cite[Theorem 1.2]{KO}.

Now let us recall an improvement of this result in the Euclidean case. More precisely, let $M=\overline{\Omega}$ with $\Omega$ an open bounded, connected and smooth open subset of $\R^n$.

We introduce also the operator $\mathcal D_{q,U,V}:\mathcal H_U(0,T')\ni H \mapsto (\partial_\nu w_{|V}, w(T,\cdot))$, with $w$ solving \eqref{eqli1}.
\begin{thm}\label{tin2} 
For $q_1,\ q_2 \in L^\infty((0,T)\times \Omega)$, the condition $\mathcal D_{q_1,U,V}=\mathcal D_{q_2,U,V}$ implies  $q_1=q_2$.
\end{thm}

This result follows from \cite[Theorem 1.1]{Ki3} combined with the definition of the trace map given in \cite[Proposition A.1]{Ki3}.

Armed with these results and the one of Theorem \ref{t3}, we will complete the proof of Theorem \ref{t1}, \ref{tt2} and \ref{tt3}.

\section{Recovery of the nonlinear terms}

The goal of this section is to combine all the tools of the preceding sections in order to complete the proof of Theorem \ref{t1}, \ref{tt2} and \ref{tt3}.

\textbf{Proof of Theorem \ref{t1}.} In view of Theorem \ref{t2},  for any $\lambda\in[-L_1,L_1]$ we have
$$\norm{\lambda\chi}_{H^{\frac{5}{2}}((0,T')\times\partial M)}\leq |\lambda|\norm{\chi}_{H^{\frac{5}{2}}((0,T')\times\partial M)}\leq L_1\norm{\chi}_{H^{\frac{5}{2}}((0,T')\times\partial M)}=L,$$
$$ \mathcal N_{F_j,\gamma}'(\lambda \chi)h=\Lambda_{q_{F_j,\lambda \chi},\gamma}h,\quad h\in\mathcal H_{*,\gamma}(0,T'),$$
where we recall that $q_{F_j,\lambda \chi}(t,x):=\partial_uF_j(t,x,u_{F_j,\lambda \chi}(t,x))$. Thus, condition \eqref{t1b} implies that $$\Lambda_{q_{F_1,\lambda \chi},\gamma,*}=\Lambda_{q_{F_2,\lambda \chi},\gamma,*}.$$ Moreover,  by the Sobolev embedding theorem and Lemma \ref{l2}, we find $$u_{F_j,\lambda \chi}\in(\mathcal C([0,T];H^2(M))\cap \mathcal C^2([0,T];L^2(M)))\subset H^2((0,T)\times M)\cap \mathcal C([0,T]\times M).$$
Combining this with the fact that  $F_j\in \mathcal C^3(\R_+\times M\times\R)$, we deduce that $q_{F_j,\lambda \chi}\in H^2((0,T)\times M)\cap \mathcal C([0,T]\times M)$ and applying Theorem \ref{t3}, we obtain
$$ q_{F_1,\lambda \chi}(t,x)=q_{F_2,\lambda \chi}(t,x),\quad (t,x,\lambda)\in(0,T)\times\gamma\times[-L_1,L_1].$$
Therefore, using the fact that $\chi=1$ on $[\delta,T]\times\gamma$, we obtain
$$ \partial_uF_1(t,x,\lambda)=q_{F_1,\lambda \chi}(t,x)=q_{F_2,\lambda \chi}(t,x)=\partial_uF_2(t,x,\lambda),\quad (t,x,\lambda)\in[\delta,T]\times\gamma\times[-L_1,L_1].$$
Finally, applying \eqref{t1a}, we obtain \eqref{t1c}. This completes the proof of Theorem \ref{t1}.\qed
\ \\
\ \\
\textbf{Proof of Theorem \ref{tt2} and \ref{tt3}.} Note first that, for  $\lambda\in[-L_2,L_2]$ and $K: (t,x)\mapsto \lambda$, we find
$$\begin{aligned}\norm{K}_{H^{\frac{5}{2}}((0,T')\times\partial M)}+\norm{K}_{H^{\frac{5}{2}}( M)}&=|\lambda|(\norm{1}_{H^{\frac{5}{2}}((0,T')\times\partial M)}+\norm{1}_{H^{\frac{5}{2}}( M)})\\
\ &\leq L_2(\norm{1}_{H^{\frac{5}{2}}((0,T')\times\partial M)}+\norm{1}_{H^{\frac{5}{2}}( M)})=L.\end{aligned}$$
Therefore, for any $\lambda\in[-L_2,L_2]$ we can fix
$$q_{F_j,\lambda}(t,x):=\partial_uF_j(t,x,u_{F_j,(\lambda,\lambda,0)}(t,x)).$$
By the Sobolev embedding theorem, we have $u_{F_j,(\lambda,\lambda,0)}\in\mathcal C([0,T]\times M)$ and we deduce that $q_{F_j,\lambda}\in\mathcal C([0,T]\times M)$. Thus, according to Theorem \ref{t2}, condition \eqref{tt2b} implies that
$$ \mathcal D_{q_{F_1,\lambda},\pd M}=\mathcal D_{q_{F_2,\lambda},\pd M},\quad  \lambda\in\left[-L_2,L_2\right].$$
Therefore, applying Theorem \ref{tin1}, we obtain
$$ q_{F_1,\lambda}(t,x)=q_{F_2,\lambda}(t,x),\quad (t,x,\lambda)\in(0,T)\times M\times\left[-L_2,L_2\right].$$
It follows that
\bel{tt2e} \partial_uF_1(0,x,\lambda)=q_{F_1,\lambda}(0,x)=q_{F_2,\lambda }(0,x)=\partial_uF_2(0,x,\lambda),\quad (x,\lambda)\in M\times\left[-L_2,L_2\right],\ee
\bel{tt2f} \partial_uF_1(t,x,\lambda)=q_{F_1,\lambda }(t,x)=q_{F_2,\lambda }(t,x)=\partial_uF_2(t,x,\lambda),\quad (t,x,\lambda)\in[0,T]\times\pd M\times\left[-L_2,L_2\right].\ee
Combining this with \eqref{tt2a} we deduce \eqref{tt2c}-\eqref{tt2d}. This proves Theorem \ref{tt2}. In a similar way, Theorem  \ref{tt3} can be deduced by combining Theorem \ref{t2} with Theorem \ref{tin2}.\qed

\section{Appendix}
This Appendix is devoted to the proof of existence of sufficiently smooth solutions of \eqref{eqn1}, with $F\in\mathcal A$ (resp. $F\in\mathcal A_*$ when $u_0=u_1=0$). 

Let us observe that we have not find any references in the mathematical literature showing existence and uniqueness of smooth solutions of \eqref{eqn1}. We have not even find  local well-posedness results for general class of seminilinear hyperbolic equations on manifolds with non-homogeneous Dirichlet boundary conditions (we have only find results  like \cite{BLP,IJ}  treating such problems with homogeneous boundary conditions). For this reason, and even though some of these results follow from well known arguments,  we have decided to provide the full proof of these results  in this Appendix. 

We start with a result of local well-posedness for the problem \eqref{eqn1} that can be proved by mean of Strichartz estimates stated in this context and several classical arguments.
\begin{lem}\label{l1} Assume that $n=2$ or $n=3$. Let $F\in \mathcal A\cup \mathcal A_*$, fix $T'\in(0,+\infty)$ and let $f\in H^{\frac{5}{2}}((0,T')\times\partial M)$, $u_0\in H^{\frac{5}{2}}( M)$ and $u_1\in H^{\frac{3}{2}}( M)$ satisfy $f_{|t=0}={u_0}_{|\partial M}$, $\partial_tf_{|t=0}={u_1}_{|\partial M}$. Let  $L>0$ be such that
\bel{lelle}\norm{f}_{H^{\frac{5}{2}}((0,T')\times\partial M)}+\norm{u_0}_{H^{\frac{5}{2}}( M)}+\norm{u_1}_{H^{\frac{3}{2}}( M)}\leq L.\ee
We consider the estimate
\bel{l1ab}\norm{u}_{\mathcal C([0,T];H^1(M))}+\norm{u}_{L^p(0,T;L^{2p}(M))}\leq C_1L,\ee
with $C_1$ depending only on $T'$, $b$, $M$, $c_1$,
and we define the sets
$$\begin{aligned}\mathcal T_{F,L}:=\{&T\in(0,T']:\ \textrm{for all data $(f,u_0,u_1)$ satisfying \eqref{lelle}, \eqref{eqn1} admits a unique solution }\\
\ &u\in\mathcal C^1([0,T]; L^2(M))\cap \mathcal C([0,T]; H^1(M))\cap L^{p}(0,T;L^{2p}(M))\textrm{ satisfying \eqref{l1ab}}\},\end{aligned}$$
$$\mathcal T_L:=\underset{F\in\mathcal A}{\bigcap}\mathcal T_{F,L},$$
with $p>1$satisfying \eqref{pp}. Then the set $\mathcal T_L$ is not empty and  $\sup\mathcal T_L=T_1(L)\in(0,T']$ depends on $L$, $b$, $M$ and $c_1$. In addition,  \eqref{eqn1}, with $T=T_1(L)$, admits a unique solution lying in  $\mathcal C^1([0,T_1(L)); L^2(M))\cap \mathcal C([0,T_1(L)); H^1(M))\cap L^{p}(0,T_1(L);L^{2p}(M))$.\end{lem}

\begin{proof} We prove this result by applying some arguments of \cite{Ka,Ki1} that we adapt to problems stated with non-homogeneous Dirichlet boundary conditions. According to \cite[Theorem 2.3, Chapter 4]{LM2}, there exists $G\in H^3((0,T')\times M)$ satisfying
$$G_{|(0,T')\times\partial M}=f,\quad G_{|t=0}=u_0,\quad \partial_tG_{|t=0}=u_1,$$
\bel{l1a}\norm{G}_{H^3((0,T')\times M)}\leq C (\norm{f}_{H^{\frac{5}{2}}((0,T')\times\partial M)}+\norm{u_0}_{H^{\frac{5}{2}}( M)}+\norm{u_1}_{H^{\frac{3}{2}}( M)})\leq CL\ee
where $C$ depends only on $M$ and $T'$. From now on and in all the remaining part of this proof, we denote by $C$ a constant depending on $M$, $\epsilon_1$, $T'$, $c_1$ and $b$.
We fix $T_2\in(0,T']$  to be determined and note that, by the Sobolev embedding theorem, we have $G\in \mathcal C([0,T_2]\times M)$ and for
$$G_1:=-(\pd_t^2 G-\Delta_{g} G ),$$
one can check that $G_1\in H^1(0,T_2;L^2(M))$. Then, we can split the solutions of  \eqref{eqn1} into two terms $u=G+v$ with $v$ solving
\bel{eqn2}
\left\{ \begin{array}{ll}  \pd_t^2 v-\Delta_{g} v + F(t,x,v+G)  =  G_1(t,x), & \mbox{in}\ (0,T_2)\times M,\\  
v  =  0, & \mbox{on}\ (0,T_2)\times \partial M,\\  
 v(0,\cdot)  =0,\quad  \pd_tv(0,\cdot)  =0 & \mbox{in}\ M.
\end{array} \right.
\ee
We will prove existence of a solution of \eqref{eqn2} by mean of a fixed point argument. We denote by $A$ the operator $-\Delta_g$ in $M$ with Dirichlet boundary condition. Now consider the map $\mathcal G$ defined on $\mathcal C([0,T_2];H^1_0(M))\cap L^p(0,T_2;L^{2p}(M))$ by
$$\mathcal G[v](t):=-\int_0^t\sin((t-s)A^{\frac{1}{2}})A^{-\frac{1}{2}}F(s,\cdot,v(s,\cdot)+G(s,\cdot))ds+\int_0^t\sin((t-s)A^{\frac{1}{2}})A^{-\frac{1}{2}}G_1(s,\cdot)ds$$
Combining the Christ-Kieslev lemma (see for instance \cite[Lemma 1]{Ki1} and also \cite{CK} for the original result) with the Strichartz estimates on manifolds stated in \cite[Theorem 1]{BSS} and following \cite[Lemma 2]{Ki1}, we deduce that
$$\begin{aligned}&\norm{\mathcal G(v)}_{\mathcal C([0,T_2];H^1(M))}+\norm{\mathcal G(v)}_{L^p(0,T_2;L^{2p}(M))}\\
&\leq  C\left(\norm{v}_{L^b(0,T_2;L^{2b}(M))}^b+\norm{G}_{L^b(0,T_2;L^{2b}(M))}^b +T_2^\frac{3}{2}L+T_2\right)\end{aligned}$$
with $C>0$ depending on $c_1$, and $b$. On the other hand, by the Sobolev embedding theorem, we have
$$\norm{G}_{L^b(0,T_2;L^{2b}(M))}\leq C\norm{G}_{L^b(0,T_2;L^{\infty}(M))}\leq C\norm{G}_{L^b(0,T_2;H^2(M))}\leq CT_2^{\frac{1}{b}}\norm{G}_{L^\infty(0,T_2;H^2(M))}\leq CT_2^{\frac{2+b}{2b}}L$$
and the H\"older inequality implies
$$\norm{v}_{L^b(0,T_2;L^{2b}(M))}\leq CT_2^{\frac{p-b}{pb}}\norm{v}_{L^p(0,T_2;L^{2p}(M))}.$$
Thus, we have
\bel{l1b}\begin{aligned}&\norm{\mathcal G(v)}_{\mathcal C([0,T_2];H^1(M))}+\norm{\mathcal G(v)}_{L^p(0,T_2;L^{2p}(M))}\\
&\leq  Cc_1T_2^{\frac{p-b}{p}}\norm{v}_{L^p(0,T_2;L^{2p}(M))}^b+Cc_1T_2^{\frac{2+b}{2}}L^b +4T_2^\frac{3}{2}L+Cc_1T_2.\end{aligned}\ee
In the same way, fixing $v_1,v_2\in \mathcal C([0,T_2];H^1_0(M))\cap L^p(0,T_2;L^{2p}(M))$ and applying  the H\"older inequality, we get
\bel{l1c}\begin{aligned}&\norm{\mathcal G(v_1)-\mathcal G(v_2)}_{\mathcal C([0,T_2];H^1(M))}+\norm{\mathcal G(v_1)-\mathcal G(v_2)}_{L^p(0,T_2;L^{2p}(M))}\\
&\leq  CT_2^{\frac{p-b}{p}}\norm{v_1-v_2}_{L^p(0,T_2;L^{2p}(M))}\left(\norm{v_1}_{L^p(0,T_2;L^{2p}(M))}^{b-1}+\norm{v_2}_{L^p(0,T_2;L^{2p}(M))}^{b-1}+\norm{G}_{L^p(0,T_2;L^{2p}(M))}^{b-1}+T_2^{\frac{b-1}{p}}\right).\end{aligned}\ee
Combining \eqref{l1b}-\eqref{l1c} with the Poincar\'e fixed point theorem, we deduce that there exists $\beta>0$ such that for
$$T_2:=\min\left(C\min\left(L^{\beta},1\right),T'\right),$$
with $C$ some suitable constant depending only on $T'$, $b$, $M$, $\epsilon_1$ and $c_1$, the map $\mathcal G$ admits a unique fixed point $v$ in the set
$$\{w\in\mathcal C([0,T_2];H^1_0(M))\cap L^p(0,T_2;L^{2p}(M)):\ \norm{w}_{\mathcal C([0,T_2];H^1(M))}+\norm{w}_{L^p(0,T_2;L^{2p}(M))}\leq C_1L\},$$
where $C_1$ is also a constant depending only on $b$, $M$ and $c_1$.
One can easily deduce that this fixed point $v$ is also lying in $\mathcal C^1([0,T_2];L^2(M))$, it satisfies \eqref{l1ab} and it solves \eqref{eqn2}. This proves the existence of local solutions for \eqref{eqn1} on $[0,T_2]$ for any $F\in \mathcal A\cup \mathcal A_*$. The uniqueness can be deduced from arguments similar to \cite[Theorem 2.1]{Ka} (see also \cite[page 134]{KY} for same ideas). This  completes the proof of the lemma.\end{proof}

This result gives us the existence and uniqueness of variational solutions of \eqref{eqn1} on $(0,T)$, provided that the conditions \eqref{lelle} and $0<T< T_1(L)$ are fulfilled. We believe that, with some suitable restrictions imposed to the set $\mathcal A$ (see for instance \cite{BLP,CH,IJ,Ka}),  this result can be extended to a global existence result corresponding to the condition $T=T_1(L)=T'$, for all $L>0$  or for some values of $L>0$. However, in the general setting, there is counterexamples to the global existence of solutions due to the blow up at finite time  of some of them (e.g. \cite[Proposition 6.4.1]{CH}). In order to preserve the generality of our results, we do not consider possible restriction of the class $\mathcal A$ of nonlinear terms which would allow  the extension of our local well-posedness result to existence of global solutions by proving that $T_1(L)=T'$, for all $L>0$ or for some values of $L>0$.

By mean of suitable conditions, we can increase the regularity of the solution $u$ of \eqref{eqn1} in the following way.

\begin{lem}\label{l2} Assume that $n=2$ or $n=3$ and, for $L>0$, fix $0<T<T_1(L)$, $F\in \mathcal A$.   Then,  for all  $(f,u_0,u_1)\in\mathcal H(0,T')$ satisfying 
$$\norm{f}_{H^{\frac{5}{2}}((0,T')\times\partial M)}+\norm{u_0}_{H^{\frac{5}{2}}( M)}+\norm{u_1}_{H^{\frac{3}{2}}( M)}\leq L,$$
   problem \eqref{eqn1} admits a unique solution lying in  $W^{1,\frac{p}{b-1}}(0,T;H^2(M))\cap W^{2,\frac{p}{b-1}}(0,T;H^1(M))$ satisfying
 \bel{l2ab}\begin{aligned}&\norm{u}_{W^{1,\frac{p}{b-1}}(0,T;H^2(M))}+\norm{u}_{W^{2,\frac{p}{b-1}}(0,T;H^1(M))}\\
&\leq C\left(\norm{f}_{H^{\frac{11}{2}}((0,T')\times\partial M)}+\norm{u_0}_{H^{\frac{11}{2}}( M)}+\norm{u_1}_{H^{\frac{9}{2}}( M)}\right),\end{aligned}\ee
with $C$ depending on $L$, $T'$, $T$, $M$, $b$, $n$ and $c_1$.
\end{lem}
\begin{proof}

According to \cite[Theorem 2.3, Chapter 4]{LM2}, in view of the compatibility condition \eqref{comp1}, there exists $G\in  H^6((0,T')\times M)$ satisfying
\bel{coco1}\begin{aligned}&G_{|(0,T')\times\partial M}=f,\quad G_{|t=0}=u_0,\quad \partial_tG_{|t=0}=u_1,\ \partial_t^2G_{|t=0}=\Delta_gu_0,\\
&\partial_t^3G_{|t=0}=\Delta_gu_1,\quad \partial_t^4G_{|t=0}=\Delta_g^2u_0,\end{aligned}\ee
$$\norm{G}_{H^3((0,T')\times M)}\leq C_1 L,$$
\bel{l2a}\norm{G}_{H^6((0,T')\times M)}\leq C_2 \left(\norm{f}_{H^{\frac{11}{2}}((0,T')\times\partial M)}+\norm{u_0}_{H^{\frac{11}{2}}( M)}+\norm{u_1}_{H^{\frac{9}{2}}( M)}\right),\ee
where $C_1$, $C_2$ depend only on $T'$, $M$.
Then, following Lemma \ref{l1}, the solution $u\in\mathcal C^1([0,T]; L^2(M))\cap \mathcal C([0,T]; H^1(M))\cap L^{p}(0,T;L^{2p}(M))$ of \eqref{eqn1} takes the form $u=v+G$ with $v\in\mathcal C^1([0,T]; L^2(M))\cap \mathcal C([0,T]; H_0^1(M))\cap L^{p}(0,T;L^{2p}(M))$ solving \eqref{eqn2}. Thus,  the proof will be completed if we prove that $v\in W^{1,\frac{p}{b-1}}(0,T;H^2(M))\cap W^{3,\frac{p}{b-1}}(0,T;L^2(M))$ satisfies
 \bel{l2ac}\norm{v}_{W^{1,\frac{p}{b-1}}(0,T;H^2(M))}+\norm{v}_{W^{2,\frac{p}{b-1}}(0,T;H^1(M))}\leq C\norm{G}_{H^6((0,T')\times M)}.\ee
For this purpose, we remark first that since $v\in\mathcal C^1([0,T]; L^2(M))\cap  L^{p}(0,T;L^{2p}(M))$, for
$$q(t,x):=\partial_uF(t,x,u(t,x)),\quad (t,x)\in[0,T]\times M,$$ 
we have
$$\norm{q}_{L^{\frac{p}{b-1}}(0,T;L^3(M))}\leq c_1\norm{1+|u|^{b-1}}_{L^{\frac{p}{b-1}}(0,T;L^3(M))}\leq C(\norm{u}_{L^p(0,T;L^{3(b-1)}(M))}^{b-1}+1)$$
and using the fact that, for $n=3$, $2p=10\geq 3(\frac{13}{3}-1)\geq3(b-1)$ and the fact that $p>3(b-1)$,  for $n=2$, we have $q\in L^{\frac{p}{b-1}}(0,T;L^3(M))$.
Thus, by the Sobolev embedding theorem, we deduce that $q\partial_tv\in L^{\frac{p}{b-1}}(0,T;H^{-1}(M))$. Moreover, using the fact that by density, for a.e $(t,x)\in(0,T)\times M$, we have
$$\partial_t[F(t,x,u(t,x))]=\partial_tF(t,x,u(t,x))+\partial_uF(t,x,u(t,x))\partial_tv(t,x)+\partial_uF(t,x,u(t,x))\partial_tG(t,x)$$and the fact that for 
$$G_2(t,x):= -\partial_uF(t,x,u(t,x))\partial_tG(t,x)-\partial_tF(t,x,u(t,x))-\partial_t^3G(t,x)+\Delta_g\partial_tG(t,x),\quad (t,x)\in[0,T]\times M,$$
we have 
$$\norm{G_2}_{L^{\frac{p}{b}}(0,T;L^2(M))}\leq C(\norm{u}_{L^{p}(0,T;L^{2p}(M))}+1),$$
we deduce that $E:(t,x)\mapsto F(t,x,u(t,x))\in W^{1,\frac{p}{b}}(0,T;H^{-1}(M))\subset\mathcal C([0,T];H^{-1}(M))$ and $v_1:=\partial_tv\in\mathcal C([0,T];L^2(M))\cap \mathcal C^1([0,T];H^{-1}(M))$.
Moreover, in view of \eqref{coco1}, we have  $$G_1(0,x)=-\partial_t^2G(0,x)+\Delta_gG(0,x)=0,\quad x\in M.$$

Therefore, combining  \cite[Theorem 9.1, Chapter 3]{LM1}  with \cite[Proposition 1]{HK}\footnote{The result \cite[Proposition 1]{HK} is stated for a bounded subdomain of $\R^n$ but it can be extended without any difficulty to a compact Riemannian manifold of dimension $n$.},  we deduce that $v_1$ is the unique element of  $\mathcal C^1([0,T]; L^2(M))\cap \mathcal C([0,T]; H_0^1(M))$ solving  the linear problem
\bel{eql1}
\left\{ \begin{array}{ll}  \pd_t^2 v_1-\Delta_{g} v_1 + q(t,x)v_1  =  G_2(t,x), & \mbox{in}\ (0,T)\times M,\\  
v_1  =  0, & \mbox{on}\ (0,T)\times \partial M,\\  
v_1(0,\cdot)  =0,\quad  \pd_tv_1(0,\cdot)  =-F(0,x,u_0(x)) & \mbox{in}\ M.
\end{array} \right.
\ee
 In the same way, we can prove that $v_2=\partial_tv_1=\partial_t^2v$ is lying in $\mathcal C^1([0,T]; L^2(M))\cap \mathcal C([0,T]; H_0^1(M))$ and it solves the linear problem
\bel{eql2}
\left\{ \begin{array}{ll}  \pd_t^2 v_2-\Delta_{g} v_2 + q(t,x)v_2  =  G_3(t,x), & \mbox{in}\ (0,T)\times M,\\  
v_2  =  0, & \mbox{on}\ (0,T)\times \partial M,\\  
 v_2(0,\cdot)  =-F(0,x,u_0(x)),\quad  \pd_tv_2(0,\cdot)  =-\partial_uF(0,x,u_0(x))u_1(x)-\partial_tF(0,x,u_0(x)) & \mbox{in}\ M,
\end{array} \right.
\ee
with
$$\begin{aligned}G_3(t,x):=& -\partial_uF(t,x,u(t,x))\partial_t^2G(t,x)-2\partial_u\partial_tF(t,x,u(t,x))[\partial_tG(t,x)+v_1(t,x)]-\partial_t^2F(t,x,u(t,x))\\
\ &-\partial_t^4G(t,x)+\Delta_g\partial_t^2G(t,x)-\partial_u^2F(t,x,u(t,x))[v_1(t,x)+\partial_tG(t,x)]^2 ,\quad (t,x)\in[0,T]\times M.\end{aligned}$$
 Here we use the fact that, by the Sobolev embedding theorem, $G\in \mathcal C^3([0,T];H^2(M))\subset \mathcal C^3([0,T];L^\infty(M))$ and  $G_3\in W^{1,\frac{p}{b}}(0,T;L^2(M))$. We use also the fact that, thanks to \eqref{comcom1}, $x\mapsto F(0,x,u_0(x))\in H^1_0(M)$. Finally, using similar arguments, we can prove that $v_3=\partial_tv_2=\partial_t^2v_1\in\mathcal C^1([0,T]; L^2(M))\cap \mathcal C([0,T]; H_0^1(M))$ solves the linear problem
\bel{eql3}
\left\{ \begin{array}{ll}  \pd_t^2 v_3-\Delta_{g} v_3 + q(t,x)v_3  =  G_4(t,x), & \mbox{in}\ (0,T)\times M,\\  
v_3  =  0, & \mbox{on}\ (0,T)\times \partial M,\\  
 v_3(0,x)  =\pd_tv_2(0,x),\quad  \pd_tv_3(0,x)  =-\Delta_{g} [F(0,x,u_0(x)] + q(0,x)F(0,x,u_0(x)+G_3(0,x), & x\in M,
\end{array} \right.
\ee
with 
$$G_4(t,x):=\partial_tG_3(t,x)-\partial_u^2F(t,x,u(t,x))(v_1(t,x)+\partial_tG(t,x))v_2(t,x).$$
Again, we use here the fact that condition \eqref{comcom1} implies
$$\partial_uF(0,x,u_0(x))=\partial_tF(0,x,u_0(x))=0,\quad x\in\partial M$$
and by the same way that $\pd_tv_2(0,\cdot)\in H^1_0(M)$. This proves that, for a.e. $t\in(0,T)$, $v_1(t,\cdot)$ solves the boundary value problem

$$\left\{ \begin{array}{ll}  -\Delta_{g} v_1(t,\cdot) =-v_3(t,\cdot)- q(t,\cdot)v_1(t,\cdot) +   G_2(t,\cdot), & \mbox{in}\ M,\\  
v_1  =  0, & \mbox{on}\ (0,T)\times \partial M\\  
\end{array} \right.$$
and using the fact that $-v_3+qv_1+G_2\in  L^{\frac{p}{b-1}}(0,T;L^2(M))$, we deduce that $v_1\in L^{\frac{p}{b-1}}(0,T;H^2(M))$. It follows that $v\in W^{1,\frac{p}{b-1}}(0,T;H^2(M))\cap W^{2,\frac{p}{b-1}}(0,T;H^1(M))$ and we obtain the required regularity result as well as \eqref{l2ac}.\end{proof}

Using similar arguments, we can prove the following.

\begin{lem}\label{ll2} Assume that $n=2$ or $n=3$ and, for $L>0$, fix $0<T<T_1(L)$, $F\in \mathcal A_*$.   Then,  for all  $f\in\mathcal H_*(0,T')$ satisfying 
$$\norm{f}_{H^{\frac{5}{2}}((0,T')\times\partial M)}\leq L,$$
   problem \eqref{eqn1}, with $u_0=u_1=0$, admits a unique solution lying in  $W^{1,\frac{p}{b-1}}(0,T;H^2(M))\cap W^{2,\frac{p}{b-1}}(0,T;H^1(M))$ satisfying \eqref{l2ab}.
\end{lem}

\section*{Acknowledgements}

 The author would like to thank Lauri Oksanen   for fruitful discussions about this problem. This work was supported by  the French National
Research Agency ANR (project MultiOnde) grant ANR-17-CE40-0029.

\end{document}